\documentclass{amsart}
\usepackage{latexsym}
\usepackage{amsmath}
\usepackage{amsthm}
\usepackage{amssymb}
\usepackage{hyperref}

\newcommand{\vip}{\vskip0.15cm}

\newcommand{\indiq}{1\!\! 1}
\newcommand{\e}{{\varepsilon}}

\newcommand{\E}{{\mathbb{E}}}

\newcommand{\be}{{\bf e}}
\newcommand{\bv}{{\bf v}}
\newcommand{\bV}{{\bf V}}
\newcommand{\bw}{{\bf w}}
\newcommand{\bW}{{\bf W}}

\newcommand{\cH}{{{\mathcal H}}}
\newcommand{\cF}{{{\mathcal F}}}
\newcommand{\cW}{{{\mathcal W}}}
\newcommand{\cL}{{{\mathcal L}}}
\newcommand{\cA}{{{\mathcal A}}}
\newcommand{\cP}{{{\mathcal P}}}

\newcommand{\rr}{{\mathbb{R}}}
\newcommand{\Sp}{{\mathbb{S}}}
\newcommand{\rd}{{\mathbb{R}^3}}

\newcommand{\intot}{\int_0^t}
\newcommand{\intrd}{\int_{\rd}}
\newcommand{\intrdd}{\int_{\rd\times\rd}}

\newcommand{\tf}{{\tilde f}}
\newcommand{\tg}{{\tilde g}}
\newcommand{\tv}{{\tilde v}}
\newcommand{\tx}{{\tilde x}}
\newcommand{\tc}{{\tilde c}}

\newcommand{\tB}{{\tilde B}}

\newcommand{\sm}{{s-}}

\newcommand{\beqn}{\begin{equation}}
\newcommand{\eeqn}{\end{equation}}
\newcommand{\bear}{\begin{eqnarray}}
\newcommand{\eear}{\end{eqnarray}}
\newcommand{\bean}{\begin{eqnarray*}}
\newcommand{\eean}{\end{eqnarray*}}

\newenvironment{preuve}{\vip\noindent {\it Proof}}{\hfill$\square$\vip}

\newtheorem{theo}{\indent Theorem}[section]
\newtheorem{prop}[theo]{\indent Proposition}
\newtheorem{rem}[theo]{\indent Remark}
\newtheorem{lem}[theo]{\indent Lemma}
\newtheorem{defin}[theo]{\indent Definition}

\begin{document}

\title[Rate of convergence of the Nanbu particle system]
{Rate of convergence of the Nanbu particle system for hard potentials and Maxwell molecules}

\author{Nicolas Fournier}
\author{St\'ephane Mischler}

\address{N. Fournier: LAMA UMR 8050, Universit\'e Paris Est,
Facult\'e de Sciences et Technologies,
61, avenue du G\'en\'eral de Gaulle, 94010 Cr\'eteil Cedex, France.}

\email{nicolas.fournier@univ-paris12.fr}

\address{S. Mischler: CEREMADE UMR 7534, Universit\'e Paris-Dauphine,
4 Place du Mar\'echal de Lattre de Tassigny F-75775, Paris Cedex 16, France.}

\email{mischler@ceremade.dauphine.fr}



\thanks{The two authors were supported  by a grant of the 
{\it Agence Nationale de la Recherche} numbered ANR-08-BLAN-0220-01.
}

\begin{abstract} 
We consider the (numerically motivated) Nanbu stochastic particle system associated to the
spatially homogeneous Boltzmann equation for true
hard potentials and Maxwell molecules. We establish a rate of propagation of chaos
of the particle system to the unique solution of the Boltzmann equation. More precisely,
we estimate the expectation of the squared 
Wasserstein distance with quadratic cost between the empirical measure
of the particle system and the solution to the Boltzmann equation. The rate we obtain is almost optimal 
as a function of the number of particles but is not uniform in time.
\end{abstract}

\maketitle

\section{Introduction and main results} 
\setcounter{equation}{0}

\subsection{The Boltzmann equation}
The Boltzmann equation predicts that 
the density $f(t,v)$ of particles with velocity $v\in \rd$ at time 
$t\geq 0$ in a spatially homogeneous dilute gas solves
\begin{eqnarray} \label{be}
\partial_t f_t(v) = \frac 1 2\intrd dv_* \int_{\Sp^{2}} d\sigma B(|v-v_*|,\theta)
\big[f_t(v')f_t(v'_*) -f_t(v)f_t(v_*)\big],
\end{eqnarray}
where the pre-collisional velocities are given by
\begin{equation}\label{vprimeetc}
v'=v'(v,v_*,\sigma)=\frac{v+v_*}{2} + \frac{|v-v_*|}{2}\sigma, \quad 
v'_*=v'_*(v,v_*,\sigma)=\frac{v+v_*}{2} -\frac{|v-v_*|}{2}\sigma
\end{equation}
and $\theta=\theta(v,v_*,\sigma)$ is the {\em deviation angle} defined by 
$\cos \theta = \frac{(v-v_*)}{|v-v_*|} \cdot \sigma$.
The {\em collision kernel} $B(|v-v_*|,\theta)\geq 0$
depends on the nature of the interactions between particles. See Cercignani \cite{C1988},
Desvillettes \cite{D2001}, Villani \cite{Vboltz2002} and Alexandre \cite{A2009} for physical and mathematical
reviews on this equation.
Conservation of mass, 
momentum and kinetic energy
hold at least formally for solutions to \eqref{be}
and we classically may assume without loss of generality that 
$\int_{\rd} f_0(v) dv=1$.

\vip

We will assume that the collision kernel is of the form
\begin{equation}\label{cs}
B(|v-v_*|,\theta)\, \sin \theta =\Phi(|v-v_*|) \, \beta(\theta) \quad \hbox{with}\quad
\beta>0 \; \hbox{ on } \; (0,\pi/2) \quad \hbox{and}\quad \beta=0 \; \hbox{ on } \;[\pi/2,\pi].
\end{equation}
This last condition $\beta=0$ on $(\pi/2,\pi]$ is not a restriction, since one can always reduce to this case
for symmetry reasons, as noted in the introduction of Alexandre {\it et al.} \cite{ADVW2000}.

\vip

When particles behave like hard spheres, it holds that $\Phi(z)=z$ and $\beta\equiv 1$.
When particles interact through a repulsive force  in $1/r^s$, 
with $s\in (2,\infty)$, one has
\begin{equation*}
\Phi(z)=z^\gamma  \;\;
\mbox{with} \;\; \gamma=\frac{s-5}{s-1}\in (-3,1)  \;\; \hbox{and} \;\;
\beta(\theta) \stackrel{0}{\sim} \mbox{cst} \, \theta^{-1-\nu} \;\;
\mbox{with} \;\;
\nu=\frac{2}{s-1}\in(0,2).
\end{equation*}
One classically names {\em hard potentials} the case when 
$\gamma\in (0,1)$  ({\em i.e.}, $s>5$ and $\nu \in (0,1/2)$), 
{\em Maxwell molecules} the case when $\gamma=0$  ({\em i.e.}, $s=5$ and $\nu =1/2$) and
{\em soft potentials} the case when $\gamma\in (-3,0)$  ({\em i.e.}, $s \in (2,5)$ and $\nu \in (1/2,2)$).
The present paper concerns Maxwell molecules, hard potentials as well as hard spheres, so that we always assume 
$\gamma \in [0,1]$.

\subsection{Stochastic particle systems}

As a step to the rigorous derivation of the Boltzmann equation, Kac \cite{Kac1956} proposed to show the 
convergence of a stochastic particle system to the solution to \eqref{be}. 
Kac's particle system is a $(\rd)^N$-valued Markov process
with infinitesimal generator $\tilde\cL_N$ defined, for
$\phi:(\rd)^N\mapsto \rr$ sufficiently regular and $\bv=(v_1,\dots,v_N) \in (\rd)^N$, by
\begin{equation*}
\tilde \cL_N \phi(\bv)= \frac 1 {2(N-1)} \sum_{i \ne j} \int_{\Sp^2} 
[\phi(\bv + (v'(v_i,v_j,\sigma)-v_i)\be_i+(v'_*(v_i,v_j,\sigma)-v_j)\be_j) - \phi(\bv)] 
B(|v_i-v_j|,\theta)d\sigma.
\end{equation*}
For $h\in\rr^3$, we note $h\be_i=(0,\dots,0,h,0,\dots,0)\in(\rd)^N$ with $h$ at the $i$-th place.
Roughly speaking, the system is constituted of $N$ particles entirely characterized by their 
velocities $(v_1,\dots,v_N)$ and
each couple of particles with velocities 
$(v_i,v_j)$ are modified, for each $\sigma\in\Sp^2$, at rate $B(|v_i-v_j|,\theta)/(2(N-1))$
and are then replaced by particles with velocities $v'(v_i,v_j,\sigma)$ and $v'_*(v_i,v_j,\sigma)$.

\vip

In the present paper, we will consider a slightly modified and non-symmetric particle system introduced by 
Nanbu \cite{N:83}. 
The Nanbu stochastic particle system corresponds to the generator 
$\cL_N$ defined, for $\phi:(\rd)^N\mapsto \rr$ sufficiently regular and
$\bv=(v_1,\dots,v_N) \in (\rd)^N$, by
\begin{equation}\label{gi}
\cL_N \phi(\bv)= \frac 1 {N} \sum_{i \ne j} \int_{\Sp^2} [\phi(\bv + (v'(v_i,v_j,\sigma)-v_i)\be_i) - \phi(\bv)] 
B(|v_i-v_j|,\theta)d\sigma.
\end{equation}
This system still describes $N$ particles characterized by their velocities $(v_1,\dots,v_N)$, but now
each couple of particles with velocities $(v_i,v_j)$ are modified, for each $\sigma\in\Sp^2$, at rate 
$B(|v_i-v_j|,\theta)/N$
and are then replaced by particles with velocities $v'(v_i,v_j,\sigma)$ and $v_j$.
Thus only one particle is modified at each ``collision'', but the rate of collision is multiplied by $2$.
All in all, the asymptotic behavior, as $N\to\infty$, should be the same.
  
\subsection{Aims} Our aim is to prove that as $N$ tends to $\infty$, the Nanbu stochastic system is 
asymptotically 
constituted of independent particles with identical law governed by the Boltzmann equation, and better, 
to quantify this convergence. 

\vip

There are two main motivations for such a study. (i) From a physical point of view, we want to know how well
the Boltzmann equation approximates true particles. Of course, true particles are subjected to classical
(non random) dynamics, so that studying the Kac (or Nanbu) particle system does not provide any rigorous
information on how well the Boltzmann equation approximates true particles. However, 
as already mentioned, Kac proposed this problem as an intermediate step.
(ii) From a numerical point of view, we want to know how well the particle system approximates
the Boltzmann equation. It is then important to get rates of convergence, 
to know how to choose the number of particles (and the cutoff parameter) to
reach a given accuracy.

%

\vip

The main difficulty lies 
in the fact that even if the particle system is initially constituted of independent particles, they do 
not remain independent for later times, because of interactions. 
Hence to answer the convergence issue, we have to prove that particles asymptotically
become independent and in the same time to identify their common law: 
we have to prove that the system is chaotic in the 
sense of Kac \cite{Kac1956}. 

\vip
We are able to prove and quantify the chaotic property for Nanbu's particle
system. 
Unfortunately, our study does really not seem to work for Kac's particle system. 
From the physical point of view, Nanbu's system
is less pertinent. However, we believe that the behaviors of the two systems are very similar, so that our 
results
should also hold true for Kac's particle system.
From the numerical point of view, both systems are expected to approximate the solution to
\eqref{be} with an error of the same order, so that the system under study is as interesting as 
Kac's system.

\vip
We will also study a cutoff version of Nanbu's system, where we remove collisions generating small deviations.
For technical reasons, we will not use the standard cutoff procedure where $B(z,\theta)$ is replaced by
$B_K(z,\theta)=B(z,\theta)\indiq_{\{\theta >1/K\}}$ for some large $K>0$. We will rather use some cutoff of the form
$B_K(z,\theta)=B(z,\theta)\indiq_{\{\theta >\varphi(K,z)\}}$, where the positive function $\varphi$ is chosen in 
such a way that $\int_{\Sp^2} B_K(z,\theta)d\sigma$ does not depend on $z$. This will simplify the argument
at several places.
This cutoff procedure is motivated by two reasons. 
From a  numerical point of view, the particle system with generator $\cL_N$ 
cannot be directly simulated, because each particle
collides with infinitely many others on each time interval (except for hard spheres). 
Thus we have to introduce a cutoff.
From a technical point of view, we are not able to prove directly our estimates for the particle system
without cutoff: we have to study first the particle system with cutoff and then to pass to the limit.

\subsection{Assumptions}
We assume that the collision kernel is of the form \eqref{cs} with
\begin{equation}\label{c1}
\exists \; \gamma \in [0,1], \; \forall z\geq 0, \; \Phi(z)=z^\gamma,
\end{equation}
and either 
\begin{equation}\label{c2hs}
\forall \theta\in(0,\pi/2),  \; \beta(\theta)=1
\end{equation}
or
\begin{equation}\label{c2}
\exists \; \nu\in (0,1), \; \exists \; 0<c_0<c_1,\; \forall \theta\in(0,\pi/2),  \;
c_0 \theta^{-1-\nu} \leq \beta(\theta) \leq c_1 \theta^{-1-\nu} .
\end{equation}
This work could probably be extended to $\nu\in(0,2)$, since the {\it important} computations 
on which it relies also hold in this case. However, this would introduce several technical
difficulties. Since Maxwell molecules and hard potentials, which we study, satisfy \eqref{c2} with
$\nu\in(0,1)$, we decided to avoid these technical complications.

\vip

The propagation of exponential moments requires the following additional condition
\begin{equation}\label{c4}
\beta(\theta)=b(\cos\theta)\quad \hbox{with $b$ non-decreasing, convex and $C^1$ on $[0,1)$.}
\end{equation}

\vip

In practice, all these assumptions are satisfied for Maxwell molecules ($\gamma=0$ and $\nu=1/2$), 
hard potentials ($\gamma\in(0,1)$ and $\nu \in(1,1/2)$) and hard spheres ($\gamma=1$ and $\beta\equiv 1$).

\subsection{Notation}

For $\theta \in (0,\pi/2)$ and $z\in [0,\infty)$ we introduce
\begin{equation}\label{defH}
H(\theta)=\int_\theta^{\pi/2} \beta(x)dx \;\; \hbox{and} \;\; G(z)=H^{-1}(z).
\end{equation}
Under \eqref{c2}, $H$ is a continuous decreasing bijection from $(0,\pi/2)$
into $(0,\infty)$, and its inverse function $G:(0,\infty)\mapsto (0,\pi/2]$ 
is defined by $G(H(\theta))=\theta$, and $H(G(z))=z$. It is immediately checked that under \eqref{c2},
there are some constants $0<c_2<c_3$ such that 
\begin{align}\label{eG}
\forall \; z >0, \quad c_2(1+z)^{-1/\nu} \leq G(z) \leq c_3 (1+z)^{-1/\nu}
\end{align}
and, as checked in \cite[Lemma 1.1]{FGu2008}, there is a constant $c_4>0$ such that
for all $x, y \in \rr_+$,
\begin{equation}\label{c3}
\int_0^\infty \left(G(z/x)-G(z/y) \right)^2 dz  
\leq c_4 \frac{(x-y)^2}{x+y}.
\end{equation}
Under \eqref{c2hs}, we have $G(z)=(\pi/2-z)_+$ (with the common notation $x_+=\max\{x,0\}$) and a
direct computation shows that \eqref{c3} also holds true.

\subsection{Well-posedness}

Let $\cP_k(\rd)$ be the set of all probability measures $f$ on $\rd$ such that
$\intrd |v|^kf(dv)<\infty$.
We first recall known well-posedness results for the Boltzmann equation, as well as some properties
of solutions we will need. A precise definition of weak solutions is stated in the next section.

\begin{theo}\label{wp}
Assume \eqref{cs}, \eqref{c1} and \eqref{c2hs} or \eqref{c2}. Let $f_0\in \cP_2(\rd)$.

(i) If $\gamma=0$, there exists a unique weak solution $(f_t)_{t\geq 0}\in C([0,\infty),\cP_2(\rd))$ 
to \eqref{be}. If $f_0\in \cP_p(\rd)$ for some $p\geq 2$, then $\sup_{[0,\infty)}
\intrd |v|^p f_t(dv) <\infty$. If $\intrd f_0(v)\log f_0(v) dv <\infty$ or if 
$f_0\in \cP_4(\rd)$ and is not a Dirac mass, then $f_t$ has a density for all $t>0$. 

(ii) If $\gamma\in (0,1]$, assume additionally \eqref{c4} and that
\begin{equation}\label{c5}
\exists \; p\in(\gamma,2),\;\;
\intrd e^{|v|^p}f_0(dv)<\infty.
\end{equation}
There is a unique weak solution $(f_t)_{t\geq 0}\in C([0,\infty),\cP_2(\rd))$ to \eqref{be} such that
\begin{equation}\label{momex}
\forall \;q\in(0,p), \quad \sup_{[0,\infty)}\intrd e^{|v|^q}f_t(dv)<\infty.
\end{equation}
Under \eqref{c2} and if $f_0$ is not a Dirac mass, then $f_t$ has a density for all $t>0$.
Under \eqref{c2hs} and if $f_0$ has a density, then $f_t$ has a density for all $t>0$.
\end{theo}

Concerning well-posedness, see Toscani-Villani \cite{TV1999} 
for Maxwell molecules, \cite{FMo2009,DM2009} for hard potentials and 
\cite{arkeryd72-2,MW1999,Lu1999,EM2010,LuM2012} 
for hard spheres.
The propagation of moments in the Maxwell case in standard, see e.g. Villani 
\cite[Theorem 1 p 74]{Vboltz2002}. The propagation of exponential moments for hard potentials and hard 
spheres, initiated by Bobylev \cite{boby97}, is checked in \cite{FMo2009,LuM2012}.
Finally, the existence of a density for $f_t$ has been proved in
\cite{F2012} (under \eqref{c2} and when $f_0$ is not a Dirac mass and belongs to $\cP_4(\rd)$),
in \cite{MW1999} (under \eqref{c2hs} when $f_0$ has a density) and is very classical
by monotonicity of the entropy when $f_0$ has a finite entropy, see e.g. Arkeryd \cite{arkeryd72-1}.

\vip

We now introduce our particle system with cutoff.

\begin{prop}\label{wps}
Assume \eqref{cs}, \eqref{c1} and \eqref{c2hs} or \eqref{c2}. Let $f_0 \in \cP_2(\rd)$ and
a number of particles $N\geq 1$ be fixed.
Let $(V^i_0)_{i=1,\dots N}$ be i.i.d. with common law $f_0$.

(i) For each cutoff parameter $K\in [1,\infty)$,
there exists a unique (in law) Markov process $(V^{i,N,K}_t)_{i=1,\dots,N,t\geq 0}$ with values
in $(\rd)^N$, starting from $(V^i_0)_{i=1,\dots N}$ and with generator $\cL_{N,K}$ defined,
for all bounded measurable $\phi:(\rd)^N\mapsto\rr$ and any $\bv=(v_1,\dots,v_N)\in\rd$,
by
\begin{align*}
\cL_{N,K} \phi(\bv)=\frac 1 {N} \sum_{i \ne j} \int_{\Sp^2} [\phi(\bv + (v'(v_i,v_j,\sigma)-v_i)\be_i) - \phi(\bv)] 
B(|v_i-v_j|,\theta)\indiq_{\{\theta \geq G(K/|v_i-v_j|^\gamma)\}}d\sigma,
\end{align*}
with $G$ defined by \eqref{defH} and, 
for $h\in\rr^3$, $h\be_i=(0,\dots,h,\dots,0)\in(\rd)^N$ with $h$ at the $i$-th place.

(ii) There exists a unique (in law) Markov process $(V^{i,N,\infty}_t)_{i=1,\dots,N,t\geq 0}$ with values
in $(\rd)^N$, starting from $(V^i_0)_{i=1,\dots N}$ and with generator $\cL_{N}$ defined, for all
Lipschitz bounded function $\phi:(\rd)^N\mapsto\rr$ and any $\bv=(v_1,\dots,v_N)\in\rd$, by \eqref{gi}.
\end{prop}

Let us emphasize that the cut-off used for defining the generator $\cL_{N,K}$ is not the usual one since 
it depends
not only on the deviation angle $\theta \in (0,2\pi)$ but also of the relative velocity $|v-v_*|$. It 
is more convenient 
in order to perform the computations we want to do. It might also be convenient for practical simulations.
Indeed, the total rate of collision of the particle system does not depend on the configuration 
of the velocities: it always equals $2\pi (N-1) K $.
Hence, the (mean) simulation cost of the particle system on a time interval $[0,T]$ is proportional to
$(N-1) KT $.

\subsection{Wasserstein distance}

For  $g,\tg \in \cP_2(\rd)$, let $\cH(g,\tg)$ be the set of probability 
measures on $\rd\times\rd$ with first marginal $g$ and second marginal $\tg$.
We then set 
\begin{eqnarray*}
\cW_2(g,\tg)&=&
\inf \left\{\left(\intrdd |v-\tv|^2 \, \eta(dv,d\tv)\right)^{1/2};\quad \eta\in \cH(g,\tg)\right\}.
\end{eqnarray*}
This is the Wasserstein distance with quadratic cost. It is well-known that the $\inf$ is reached.
We refer to Villani \cite[Chapter 2]{Vtot2003} 
for more details on this distance.
A remarkable result,
due to Tanaka \cite{T1978,T1979}, is that in the case of Maxwell molecules,
$t\mapsto \cW_2(f_t,\tf_t)$ is non-increasing 
for each pair of reasonable solutions $f,\tf$ to the Boltzmann equation.
The present work is strongly inspired by the ideas of Tanaka.

\subsection{Empirical law of large numbers}

For $f\in \cP_2(\rd)$ and $N\geq 1$, we define
\begin{equation}\label{bestrate}
\e_N(f):= \E \left[\cW_2^2\left(f,N^{-1}\sum_1^N \delta_{X_i} \right)\right] \hbox{ with $X_1,\dots,X_N$ 
independent and $f$-distributed.}
\end{equation}

Since a $f$-chaotic stochastic particle system is asymptotically constituted of 
i.i.d. $f$-distributed particles, $\e_N(f)$ is the best rate (as far as $\cW_2^2$ is concerned) 
we can hope for such a system. We recall now the estimate proved in \cite[Theorem 1]{fg} (with $d=3$ and $p=2$) 
and we also refer to Rachev-Ruschendorf \cite[Theorem10.2.1]{RR},   \cite[Lemma 4.2]{MMKac}, Boissard-Le Gouic
\cite{BoissardLeGouic} and  Dereich-Scheutzow-Schottstedt \cite{dss} for earlier (but not optimal) versions.

\begin{theo}\label{theo:W2indep}
For all $A>0$, all $k > 2$, all $f\in\cP_k(\rd)$ verifying
$\intrd |v|^k f(dv) \leq A$, all $N\geq 1$,
\beqn\label{eq:W2indep}
\e_N(f)\leq \left\{ 
\begin{array}{ll}
C_{A,k} N^{-(k-2)/k} & \hbox{ if $k \in (2,4)$,}\\
C_{A,k} N^{-1/2} & \hbox{  if $k>4$}.
\end{array}\right.
\eeqn
\end{theo}

This bound is optimal for {\it general} laws. The convergence might be faster for some
regular laws, but this should be quite complicated, see \cite[Subsection 1.2]{fg} as well as the 
discussion in Barthe-Bordenave \cite{BartheBordenave}. We also refer to 
\cite[Theorem~2.13]{HaurayMischler} (and the remarks which follow) for a general discussion 
about the rate of chaoticity 
for independent and dependent random arrays.


\subsection{Main result}

Our study concerns both the particle systems with and without cutoff. 
It is worth to notice that for true Maxwell molecules and hard potentials,
$\nu \in (0,1/2]$ so that $1-2/\nu \le -3$ and the contribution of the cut-off approximation vanishes
rapidly in the limit $K\to\infty$. 

\begin{theo}\label{mr} 
Let $B$ be a collision kernel satisfying \eqref{cs}, \eqref{c1} and \eqref{c2hs} or \eqref{c2} and let 
$f_0\in\cP_2(\rd)$ not be a Dirac mass.
If $\gamma>0$, assume additionally \eqref{c4} and \eqref{c5}.
Consider the unique weak solution $(f_t)_{t\geq 0}$ to \eqref{be} defined in Theorem \ref{wp}
and, for each $N\geq 1$, $K\in [1,\infty]$, 
the unique Markov process $(V^{i,N,K}_t)_{i=1,\dots,N,t\geq 0}$  defined in Proposition \ref{wps}.
Let $\mu^{N,K}_t:= N^{-1}\sum_1^N \delta_{V^{i,N,K}_t}$.

(i) Maxwell molecules. Assume that $\gamma=0$, \eqref{c2} 
and either $\intrd f_0(v)\log f_0(v) dv<\infty$ or $f_0\in\cP_4(\rd)$.
There is a constant $C$ such that for all $T\geq 0$,
all $N \geq 1$, all $K\in[1,\infty]$,
\begin{align}\label{fc1}
\sup_{[0,T]}\E[\cW_2^2(\mu^{N,K}_t,f_t)] 
\leq& C (1+T)^2 \sup_{[0,T]} \e_N(f_t)+  C T K^{1-2/\nu}.
\end{align}
If $f_0\in\cP_k(\rd)$ for some $k>2$, we have $\sup_{[0,\infty)} \intrd |v|^k f_t(dv)<\infty$
and we can use Theorem \ref{theo:W2indep} to bound $\sup_{[0,T]} \e_N(f_t)$.
In particular if $k>4$, then for all $T\geq 0$,
all $N \geq 1$, all $K\in[1,\infty]$,
\begin{align}\label{fc2}
\sup_{[0,T]}\E[\cW_2^2(\mu^{N,K}_t,f_t)] \leq& C (1+T)^2 N^{-1/2}+  C T K^{1-2/\nu}.
\end{align}

(ii) Hard potentials. Assume that $\gamma\in(0,1)$ and \eqref{c2}. For all $\e\in(0,1)$, all $T\geq0$, 
there is a constant $C_{\e,T}$ such that for all $N\geq 1$, all $K\in[1,\infty]$,
\begin{align}\label{fc3}
\sup_{[0,T]}\E[\cW_2^2(\mu^{N,K}_t,f_t)] \leq& C_{\e,T} \left(\sup_{[0,T]}\e_N(f_t)
+ K^{1-2/\nu} \right)^{1-\e}.
\end{align}
Consequently,  for all $\e\in(0,1)$, all $T\geq0$, there is 
$C_{\e,T}$ such that for all $N\geq 1$, all $K\in[1,\infty]$,
\begin{align}\label{fc4}
\sup_{[0,T]}\E[\cW_2^2(\mu^{N,K}_t,f_t)] \leq& C_{\e,T} (N^{-1/2} + K^{1-2/\nu})^{1-\e}. 
\end{align}

(iii) Hard spheres. Assume finally that $\gamma=1$,  \eqref{c2hs} and that $f_0$ has a density.
For all $\e\in(0,1)$, all $T\geq0$, all $q\in (1,p)$, 
there is a constant $C_{\e,q,T}$ such that for all $N\geq 1$, all $K\in[1,\infty)$,
\begin{align}\label{fc5}
\sup_{[0,T]}\E[\cW_2^2(\mu^{N,K}_t,f_t)] \leq& C_{\e,q,T} \left(\left(\sup_{[0,T]}\e_N(f_t)\right)^{1-\e}
+ e^{-K^q}\right)e^{C_{\e,q,T} K} .
\end{align}
Thus for all $\e\in(0,1)$, all $T\geq0$, all $q\in(1,p)$, there is 
$C_{\e,q,T}$ such that for all $N\geq 1$, all $K\in[1,\infty)$,
\begin{align}\label{fc6}
\sup_{[0,T]}\E[\cW_2^2(\mu^{N,K}_t,f_t)] \leq& C_{\e,q,T} (N^{-1/2+\e} +  e^{-K^q})e^{C_{\e,q,T} K}    . 
\end{align}
\end{theo}

Concerning the rate of convergence of the simulation algorithm, we have the following.

\begin{rem}
Recall that the simulation cost per unit of time is proportional to $(N-1)K$.

(i) For Maxwell molecules and hard potentials the error (for $\cW_2$) 
is $(N^{-1/6}+K^{1/2-1/\nu})^{1-}$. For
a given simulation cost $\tau$, the best choices are $N\simeq \tau^{(4-2\nu)/(4-\nu)}$ and 
$K\simeq \tau^{\nu/(4-\nu)}$, 
which leads to an error in $\tau^{-(2-\nu)/(8-2\nu) +}$. For true hard potentials
and Maxwell molecules, this is at worst $\tau^{-3/14+}$ and at best $\tau^{-1/4+}$.

(ii) For hard spheres, make the choice $K\simeq (\log N)^a$ with $a \in (1/q,1)$.
Then $e^{C K} << N^{\e}$ for any $\e\in (0,1)$ and $e^{-K^q}<< N^{-r}$ for any $r>1$.
With this choice, we thus find an error in $N^{-1/4+\e}$ for a simulation cost in $N  (\log N)^a$.
Consequently, for a given simulation cost $\tau$, we find an error in $\tau^{-1/4 +}$.
\end{rem}

We excluded the case where $f_0$ is a Dirac mass because we need that $f_t$ has a density and because
if $f_0=\delta_{v_0}$, then the unique solution to \eqref{be} is given by $f_t=\delta_{v_0}$ and 
the Markov process of Proposition \ref{wps} is nothing but $V^{1,N,K}_t=(v_0,\dots v_0)$ 
(for any value of $K\in [1,\infty]$), so that $\mu^{N,K}_t=\delta_{v_0}$ and thus $\cW_2(f_t,\mu^{N,K}_t)=0$.

\subsection{Comments}
We thus show that the empirical law of the particle system converges
to $f_t$ as fast as i.i.d. $f_t$-distributed particles (up to an arbitrary small loss if $\gamma\ne 0$).
This is thus almost optimal in some sense. However, this is optimal only as far
as $\cW_2$ is concerned: we would have preferred to work with another distance and to obtain
a rate in $N^{-1/2}$ as is expected for laws of large numbers. Here we obtain a rate in $N^{-1/4}$,
since $\cW_2$ is squared. However, $\cW_2$ enjoys several properties that make it quite convenient
when studying the Boltzmann equation, mainly because of the role of the kinetic energy.
Another default of this work is that we obtain a non-uniform (in time) bound. For Maxwell
molecules, the bound is slowly increasing (as $T^2$) but for hard potentials, it is growing very fast.

\vip

Note also that for hard spheres, we are not able to treat the case where $K=\infty$:
we need to let $K$ and $N$ go to infinity simultaneously, with some constraints.
We believe that this is only a technical problem, but we were not able to solve it. However, we still obtain
a very reasonable rate of convergence (as a function of the computational cost).

\vip

Our proof is based on a coupling argument: we couple the $N$-particle system
with a family of $N$ i.i.d. Boltzmann processes,
in such a way that they remain as close as possible.  
We prove an accurate control on 
the increment of the distance between the two systems at each collision. 
This last computation is similar to those of \cite{FMo2009,FGu2008} 
concerning uniqueness of the solution to \eqref{be}.
However, we need to handle much more precise computations: in \cite{FMo2009}, when studying the distance 
between two solutions to \eqref{be}, both were supposed to have exponential moments.
Such exponential moments are known to propagate for solutions to \eqref{be} since
the seminal work of Bobylev \cite{boby97}, but for 
the particle system under study, we are not even able to prove the finiteness of a moment of order $2+\e$, $\e>0$! 
We thus need a very precise refinement of the computations of \cite{FMo2009,FGu2008}.

\vip

All these problems do not appear when studying Maxwell molecules. Roughly, the 
collision operator is globally Lipschitz continuous for Maxwell molecules and only locally
Lipschitz continuous for hard potentials (which explains why large velocities have to be controlled by
using exponential moments). This is why we obtain a better result for Maxwell molecules.

\vip

Note that for the (physically more relevant) Kac particle system
moments are known to propagate (uniformly in $N$), see Sznitman \cite{SznitmanB} and also \cite{MMKac}, 
which would simplify greatly the proof at many places. However,
we are not able to exhibit a suitable coupling. This is due to the fact that in Kac's
system, each collision modifies the velocity of two particles. 
In Nanbu's system, the Poisson measures governing two different particles are independent,
which is not the case for Kac's system (because each time a particle's velocity is modified, another one has 
to be also modified) although the larger is the number of particles, the lower the correlation is. 
As a consequence, it is more difficult
to couple the $N$-particle symmetric Kac's system with $N$ independent copies of the Boltzmann process
and we did not succeed.

\subsection{Known results}
Such a chaos result for the Boltzmann equation with bounded cross section, or for related models, has been 
first established without any rate by Kac \cite{Kac1956} (for the so-called Maxwell molecules Kac's 
model which is roughly a 
``toy one-dimensional" Boltzmann equation) and then by 
McKean   \cite{McKean1967} and Gr\"unbaum \cite{Grunbaum}. For unbounded cross section, the chaos property 
has been proved by Sznitman \cite{SznitmanB} for hard spheres, still without rate.

\vip

For Maxwell molecules with Grad's cutoff, a nice rate of convergence 
(of order $1/N$ in total variation distance on the two-marginal)
has been obtained by McKean \cite{McK4} and improved by Graham-M\'el\'eard \cite{GM}. 
This was extended by Desvillettes-Graham-M\'el\'eard \cite{DGM1999}, see also \cite{FMe2002sto},
to true (without Grad's cutoff) Maxwell molecules, 
but with a rate in $N^{-1}e^{KT}+K^{1-2/\nu}$ (with the notation of
the present paper). From a numerical point of view, this leads to a logarithmic convergence as a function
of the computational cost.

\vip

More recently, 
a uniform in time rate of chaos convergence of Kac's stochastic particle system 
to the Boltzmann equation for two unbounded models has been established in 
\cite{MMKac,KleberBsphere} (see also \cite{MMWchaos}), by taking up again and improving Gr\"unbaum's 
approach. 
For true Maxwell molecules, uniform in time rate of convergence of order $N^{-1/(6+\delta)}$, for any 
$\delta >0$,
for a weak distance on the two-marginals has been proved in \cite[Theorem 5.1]{MMKac} when the initial 
condition $f_0$
has a compact support. This result was improved and made more precise in \cite[Step 3 of the proof of 
Theorem 8]{KleberBsphere}, where, still for true Maxwell molecules, uniform in time rate of  convergence of 
order $N^{-1/177}$, for the same $\cW_2$ Wasserstein distance as used in \eqref{fc1}, 
has been proved for any initial condition $f_0$ satisfying \eqref{c5}. Hard spheres have also been
studied in \cite[Theorem 6.1]{MMKac}: a uniform
in time rate of convergence of order $1/(\log N)^\alpha$ with $\alpha > 0$ small, for the $\cW_1$ 
distance on the two-marginals has been proved. When applying the methods of 
\cite{MMKac,MMWchaos,KleberBsphere} on finite time intervals, the previous rates can not be 
really improved.
Finally, let us mention that the present work follows some of the ideas of \cite{FGo2012}, which concerns
the Kac equation.

\vip

To summarize: 

$\bullet$ We obtain the first rate of convergence for hard potentials and this rate is reasonable. 
Recall that hard potentials are twice unbounded (the velocity cross section is unbounded
and the angular cross section is non-integrable), while Maxwell molecules enjoy a bounded velocity cross
section and hard spheres an integrable angular cross section.

$\bullet$ For hard spheres
and Maxwell molecules, we prove a much faster convergence than \cite{MMKac,MMWchaos,KleberBsphere},
but we are restricted to finite time-intervals and we cannot study Kac's system.

\vip

Let us finally mention that we use a coupling method, as is widely used since the 
famous {\it cours \`a l'\'ecole d'\'et\'e de Saint-Flour} 
by Sznitman \cite{Sbook} for providing rate of chaos convergence 
for the so-called McKean-Vlasov model and that such methods 
have been recently adapted to non-globally Lipschitz
coefficients by Bolley-Ca\~nizo-Carrillo in \cite{BCC}, making use of exponential
moments.

\subsection{Plan of the paper}
In Section \ref{prel}, we make precise the notion of weak solutions, rewrite 
the collision operators in a suitable form and check a accurate version of a lemma due to Tanaka \cite{T1979}.
Section \ref{sec:MainComput} is devoted to the cornerstone estimate on the collision integral. 
In Section \ref{sec:ConvCutoff}  we prove the convergence of the particle system with cutoff. The
cutoff is removed in Section~\ref{sec:ConvGal}.

\section{Preliminaries}\label{prel}
\setcounter{equation}{0}

\subsection{Rewriting equations}

We follow here \cite{FMe2002sto}.
For each $X\in \rd$, we introduce $I(X),J(X)\in\rd$ such that
$(\frac{X}{|X|},\frac{I(X)}{|X|},\frac{J(X)}{|X|})$ 
is a direct orthonormal basis of $\rd$ and, of course, in such a way that $I,J$ are measurable functions.
For $X,v,v_*\in \rd$, for $\theta \in (0,\pi/2)$ and $\varphi\in[0,2\pi)$,
we set
\begin{equation}\label{dfvprime}
\left\{
\begin{array}{l}
\Gamma(X,\varphi):=(\cos \varphi) I(X) + (\sin \varphi)J(X), \\\\
a(v,v_*,\theta,\varphi):= - \displaystyle\frac{1-\cos\theta}{2} (v-v_*)
+ \frac{\sin\theta}{2}\Gamma(v-v_*,\varphi),\\\\
v'(v,v_*,\theta,\varphi):=v+a(v,v_*,\theta,\varphi),
\end{array}
\right.
\end{equation}
which is a suitable parametrization of \eqref{vprimeetc}:
write $\sigma\in \Sp^2$ as $\sigma=\frac{v-v_*}{|v-v_*|}\cos\theta
+ \frac{I(v-v_*)}{|v-v_*|}\sin\theta \cos\varphi+\frac{J(v-v_*)}{|v-v_*|}\sin\theta
\sin\varphi$.
Let us define, classically, weak solutions to \eqref{be}.

\begin{defin}\label{dfsol}
Assume \eqref{cs}, \eqref{c1} and \eqref{c2hs} or \eqref{c2}.
A family $(f_t)_{t\geq 0} \in C([0,\infty),\cP_2(\rd))$ is called a weak solution to \eqref{be} if
it preserves momentum and energy, i.e.
\begin{equation}\label{cons}
\forall\; t\geq 0, \quad \intrd v f_t(dv)= \intrd v f_0(dv) \quad\hbox{and}\quad
\quad \intrd |v|^2 f_t(dv)= \intrd |v|^2 f_0(dv)
\end{equation}
and if for any $\phi:\rd\mapsto \rr$  bounded and Lipschitz-continuous, any $t\in [0,T]$,
\begin{equation}\label{wbe}
\intrd \phi(v)\, f_t(dv) =  \intrd \phi(v)\, f_0(dv)
+\intot \intrd \intrd \cA\phi(v,v_*) f_s(dv_*)   f_s(dv) ds
\end{equation}
where
\begin{equation}\label{afini}
\cA\phi (v,v_*) = |v-v_*|^\gamma \, \int_0^{\pi/2}
\beta(\theta)d\theta 
\int_0^{2\pi} d\varphi
\left[\phi(v+a(v,v_*,\theta,\varphi))-\phi(v) \right].
\end{equation}
\end{defin}

Noting that $|a(v,v_*,\theta,\varphi)| \leq C  \theta |v-v_*|$ and that $\int_0^{\pi/2}
\theta \beta(\theta)d\theta$,
we easily get $|\cA\phi(v,v_*)|\leq C_\phi |v-v_*|^{1+\gamma} \leq C_\phi  (1+|v-v_*|^2)$,
so that everything makes sense in \eqref{wbe}.

\vip

We next rewrite the collision
operator in a way that makes disappear the velocity-dependence
$|v-v_*|^\gamma$ in the {\it rate}. Such a trick 
was already used in \cite{FMe2002weak} and \cite{FGu2008}.

\begin{lem}\label{rewriteA} 
Assume \eqref{cs}, \eqref{c1} and \eqref{c2hs} or \eqref{c2}.
Recalling  \eqref{defH} and \eqref{dfvprime}, define,
for $z\in (0,\infty)$,
$\varphi\in [0,2\pi)$, $v,v_*\in \rd$ and $K\in [1,\infty)$, 
\begin{equation}\label{dfc}
c(v,v_*,z,\varphi):=a[v,v_*,G(z/|v-v_*|^\gamma),\varphi] \;\hbox{ and }\;
c_K(v,v_*,z,\varphi):=c(v,v_*,z,\varphi)\indiq_{\{z \leq K\}}.
\end{equation} 
For any bounded Lipschitz $\phi:\rd\mapsto\rr$,  any $v,v_*\in \rd$
\begin{eqnarray}\label{agood}
\cA\phi(v,v_*)&=&\int_0^\infty dz \int_0^{2\pi}d\varphi
\Big(\phi[v+c(v,v_*,z,\varphi)] -\phi[v]\Big).
\end{eqnarray}
For any $N\geq 1$, $K\in[1,\infty)$, $\bv=(v_1,\dots,v_N)\in(\rd)^N$, any bounded measurable
$\phi:(\rd)^N\mapsto\rr$,
\begin{equation}\label{lKgood}
\cL_{N,K} \phi(\bv)= \frac 1 {N} \sum_{i \ne j} \int_0^\infty dz\int_0^{2\pi}d\varphi 
[\phi(\bv + c_K(v_i,v_j,z,\varphi)\be_i) - \phi(\bv)].
\end{equation}
For any $N\geq 1$, any $\bv=(v_1,\dots,v_N)\in(\rd)^N$, any bounded Lipschitz
$\phi:(\rd)^N\mapsto\rr$,
\begin{equation}\label{lgood}
\cL_{N} \phi(\bv)= \frac 1 {N} \sum_{i \ne j} \int_0^\infty dz\int_0^{2\pi}d\varphi 
[\phi(\bv + c(v_i,v_j,z,\varphi)\be_i) - \phi(\bv)].
\end{equation}
\end{lem}

\begin{proof} To get \eqref{agood}, start from \eqref{afini} and use the substitution 
$\theta=G(z/|v-v_*|^\gamma)$ or equivalently $H(\theta) = z/|v-v_*|^\gamma$, which implies 
$|v-v_*|^\gamma\beta(\theta)d\theta=dz$. 
The expressions \eqref{lKgood} and \eqref{lgood} are checked similarly.
\end{proof}

\subsection{Accurate version of Tanaka's trick}

As was already noted by Tanaka \cite{T1979}, it is not possible to choose
$I$ in such a way that $X\mapsto I(X)$ is continuous. However, he 
found a way to overcome this difficulty, see also 
\cite[Lemma 2.6]{FMe2002sto}. Here we need the following accurate version of Tanaka's trick.

\begin{lem}\label{tanana}
Recall \eqref{dfvprime}. There are some measurable functions $\varphi_0,\varphi_1 : \rd \times \rd
\mapsto [0,2\pi)$, such that for all $X,Y\in\rd$, all $\varphi\in[0,2\pi)$,
\begin{align*}
&\Gamma(X,\varphi)\cdot\Gamma(Y, \varphi + \varphi_0(X,Y))=X\cdot Y \cos^2(\varphi+\varphi_1(X,Y) )
+ |X||Y|\sin^2(\varphi+\varphi_1(X,Y)),\\
&|\Gamma(X,\varphi) - \Gamma(Y, \varphi + \varphi_0(X,Y))|\leq |X - Y|.
\end{align*}
\end{lem}

\begin{proof} First observe that the second claim follows from the first one: writing 
$\varphi_i=\varphi_i(X,Y)$
\begin{align*}
|\Gamma(X,\varphi) - \Gamma(Y, \varphi + \varphi_0)|^2
=& |\Gamma(X,\varphi)|^2+
|\Gamma(Y, \varphi + \varphi_0)|^2 -2 \Gamma(X,\varphi)\cdot\Gamma(Y, \varphi + \varphi_0)\\
=&|X|^2+|Y|^2-2(X\cdot Y \cos^2(\varphi+\varphi_1) + |X||Y|\sin^2(\varphi+\varphi_1))\\
\leq & |X|^2+|Y|^2-2X\cdot Y=|X-Y|^2.
\end{align*}
We next check the first claim.
Let thus $X$ and $Y$ be fixed. Observe that $\Gamma(X,\varphi)$ goes (at constant speed) 
all over the circle $C_X$  with radius $|X|$ lying in the plane orthogonal to $X$.
Let $i_X \in C_X$ and $i_Y\in C_Y$ such that $X,Y,i_X,i_Y$ belong to the same plane and $i_X\cdot i_Y=X\cdot Y$
(there are exactly two possible choices for the couple $(i_X,i_Y)$ if $X$ and $Y$ are not collinear, 
infinitely many otherwise). Consider $\varphi_X$  and $\varphi_Y$
such that $i_X:=\Gamma(X,\varphi_X)$ and $i_Y:=\Gamma(Y,\varphi_Y)$.
Define $j_X:=\Gamma(X,\varphi_X+\pi/2)$ and 
$j_Y:=\Gamma(Y,\varphi_Y+\pi/2)$. Then $j_X$ and $j_X$ are collinear 
(because both are orthogonal to the plane containing $X,Y,i_X,i_Y$), satisfy
$j_X \cdot j_Y=|j_X||j_Y|=|X||Y|$ and $i_X \cdot j_Y=i_Y.j_X=0$. 
Next, observe that $\Gamma(X,\varphi+\varphi_X)=i_X\cos\varphi + j_X \sin\varphi$
while $\Gamma(Y,\varphi+\varphi_Y)=i_Y\cos\varphi + j_Y \sin\varphi$. Consequently,
$\Gamma(X,\varphi+\varphi_X)\cdot\Gamma(Y,\varphi+\varphi_Y)= i_X\cdot i_Y \cos^2\varphi + 
j_X \cdot j_Y \sin^2\varphi
= X\cdot Y \cos^2\varphi + |X||Y|\sin^2\varphi$.
The conclusion follows: choose $\varphi_0:=\varphi_Y-\varphi_X$ and $\varphi_1:=-\varphi_X$ 
(all this modulo $2\pi$).
\end{proof}

\section{Main computations of the paper}\label{sec:MainComput}
\setcounter{equation}{0} 

The following estimate is our central argument.

\begin{lem}\label{fundest}
Recall that $G$ was defined in \eqref{defH} and that the deviation functions $c$ and $c_K$ were defined in 
\eqref{dfc}. For any $v,v_*,\tv,\tv_* \in \rd$, any $K\in [1,\infty)$,
\begin{align*}
&\int_0^\infty \int_0^{2\pi} \Big( 
\big|v+c(v,v_*,z,\varphi)-\tv-c_K(\tv,\tv_*,z,\varphi+\varphi_0(v-v_*,\tv-\tv_*)) \big|^2 - |v-\tv|^2
\Big) d\varphi   dz\\
\leq& A_1^K(v,v_*,\tv,\tv_*)+ A_2^K(v,v_*,\tv,\tv_*)+A_3^K(v,v_*,\tv,\tv_*),
\end{align*}
where, setting 
$\Phi_K(x)=\pi \int_0^K (1-\cos G(z/x^\gamma))dz$ and $\Psi_K(x)=\pi \int_K^\infty (1-\cos G(z/x^\gamma))dz$,
\begin{align*}
A_1^K(v,v_*,\tv,\tv_*)= &2  |v-v_*||\tv-\tv_*| 
\int_0^K \big[G(z/|v-v_*|^\gamma)-G(z/|\tv-\tv_*|^\gamma)\big]^2 dz,\\
A_2^K(v,v_*,\tv,\tv_*)= &-  \big[(v-\tv)+(v_*-\tv_*)\big]\cdot\big[(v-v_*)
\Phi_K(|v-v_*|)-(\tv-\tv_*)\Phi_K(|\tv-\tv_*|)\big],\\
A_3^K(v,v_*,\tv,\tv_*)= &(|v-v_*|^2+2|v-\tv||v-v_*|)\Psi_K(|v-v_*|).
\end{align*}
\end{lem}

\begin{proof}
We need to shorten notation.
We write $x=|v-v_*|$, $\tx=|\tv-\tv_*|$, $\varphi_0=\varphi_0(v-v_*,\tv-\tv_*)$,
$c=c(v,v_*,z,\varphi)$, $\tc=c(\tv,\tv_*,z,\varphi+\varphi_0)$ and
$\tc_K=c_K(\tv,\tv_*,z,\varphi+\varphi_0)=\tc\indiq_{\{z\leq K\}}$. We start with
\begin{align*}
\Delta_K:=& \int_0^\infty \int_0^{2\pi} \Big(|v+c - \tv-\tc_K|^2-|v-\tv|^2 \Big) d\varphi dz \\
=& \int_0^K \int_0^{2\pi} \Big( |c|^2+|\tc|^2-2c\cdot\tc +2(v-\tv)\cdot(c-\tc)\Big)  d\varphi dz\\
&+\int_K^\infty \int_0^{2\pi} \Big( |c|^2 +2(v-\tv)\cdot c\Big)  d\varphi dz.
\end{align*}

First, it holds that $|c|^2=|-(1-\cos G(z/x^\gamma))(v-v_*)+(\sin G(z/x^\gamma)) \Gamma(v-v_*,\varphi)|^2/4
=(1-\cos G(z/x^\gamma))|v-v_*|^2/2$.
We used that by definition, see \eqref{dfvprime}, $\Gamma(v-v_*,\varphi)$ has the same norm as $v-v_*$
and is orthogonal to $v-v_*$ and that $(1-\cos \theta)^2+(\sin \theta)^2=2-2\cos \theta$. Consequently,
we have
\begin{align*}
\int_0^K \int_0^{2\pi} |c|^2 d\varphi dz = \pi |v-v_*|^2 \int_0^K (1-\cos G(z/x^\gamma)) dz =x^2  \Phi_K(x).
\end{align*}
Similarly, we also have $\int_0^K \int_0^{2\pi} |\tc|^2 d\varphi dz = \tx^2  \Phi_K(\tx)$
and $\int_K^\infty \int_0^{2\pi} |c|^2 d\varphi dz = x^2  \Psi_K(x)$.

\vip

Next, using that $c=-(1-\cos G(z/x^\gamma))(v-v_*)/2+(\sin G(z/x^\gamma)) \Gamma(v-v_*,\varphi)/2$ and that 
$\int_0^{2\pi} \Gamma(v-v_*,\varphi) d\varphi=0$,
\begin{align*}
\int_0^K \int_0^{2\pi} c d\varphi dz= -(v-v_*) \pi \int_0^K (1-\cos G(z/x^\gamma))dz 
= - (v-v_*) \Phi_K(x).
\end{align*}
By the same way, $\int_0^K \int_0^{2\pi} \tc d\varphi dz= - (\tv-\tv_*) \Phi_K(\tx)$
and $\int_K^\infty \int_0^{2\pi} c d\varphi dz= - (v-v_*) \Psi_K(x)$.

\vip

Finally, $c\cdot\tc=[(1-\cos G(z/x^\gamma))(v-v_*)-(\sin G(z/x^\gamma))\Gamma(v-v_*,\varphi)]
\cdot[(1-\cos G(z/\tx^\gamma))(\tv-\tv_*)-(\sin G(z/\tx^\gamma))
\Gamma(\tv-\tv_*,\varphi+\varphi_0)]/4$. Since $\int_0^{2\pi} \Gamma(v-v_*,\varphi) d\varphi=
\int_0^{2\pi} \Gamma(\tv-\tv_*,\varphi+\varphi_0) d\varphi=0$, we get
\begin{align*}
\int_0^{2\pi} c\cdot\tc d\varphi =& \frac\pi 2 (1-\cos G(z/x^\gamma))(1-\cos G(z/\tx^\gamma)) 
(v-v_*)\cdot(\tv-\tv_*)\\
&+\frac 1 4 
(\sin G(z/x^\gamma))(\sin G(z/\tx^\gamma)) 
\int_0^{2\pi}\Gamma(v-v_*,\varphi)\cdot \Gamma(\tv-\tv_*,\varphi+\varphi_0)  d\varphi.
\end{align*}
Recalling Lemma \ref{tanana} and using that $\int_0^{2\pi} \cos^2(\varphi+\varphi_1) d\varphi=
\int_0^{2\pi} \sin^2(\varphi+\varphi_1) d\varphi = \pi$, we obtain
\begin{align*}
\int_0^{2\pi} c\cdot\tc d\varphi =& \frac\pi 2 (1-\cos G(z/x^\gamma))(1-\cos G(z/\tx^\gamma)) 
(v-v_*)\cdot(\tv-\tv_*)\\
&+ \frac \pi 4 (\sin G(z/x^\gamma))(\sin G(z/\tx^\gamma)) \big[(v-v_*)\cdot(\tv-\tv_*)+|v-v_*||\tv-\tv_*| \big].
\end{align*}
But $G$ takes values in $(0,\pi/2)$, so that, since $|v-v_*||\tv-\tv_*|\geq (v-v_*)\cdot(\tv-\tv_*)$,
\begin{align*}
&\int_0^{2\pi} c\cdot\tc d\varphi \\
\geq&  \frac\pi 2 [(1-\cos G(z/x^\gamma))(1-\cos G(z/\tx^\gamma))
+(\sin G(z/x^\gamma))(\sin G(z/\tx^\gamma))] (v-v_*)\cdot(\tv-\tv_*)
\\
=& \frac\pi 2 [(1-\cos G(z/x^\gamma))+(1-\cos G(z/\tx^\gamma))](v-v_*)\cdot(\tv-\tv_*)\\
& -\frac\pi 2(1-\cos(G(z/x^\gamma)-G(z/\tx^\gamma)))(v-v_*)\cdot(\tv-\tv_*).
\end{align*}
Using that $\pi(1-\cos \theta)\leq 2 \theta^2$, we thus get 
\begin{align*}
\int_0^K \int_0^{2\pi} c\cdot\tc d\varphi dz 
\geq & (v-v_*)\cdot(\tv-\tv_*)\frac{\Phi_K(x)+\Phi_K(\tx)}2
-  x \tx  \int_0^K (G(z/x^\gamma)-G(z/\tx^\gamma))^2 dz.
\end{align*}

\vip

All in all, we find
\begin{align*}
\Delta_K \leq & x^2\Phi_K(x)+\tx^2\Phi_K(\tx)-(v-v_*)\cdot(\tv-\tv_*)[\Phi_K(x)+\Phi_K(\tx) ]\\
&+ 2(v-\tv)\cdot[(\tv-\tv_*)\Phi_K(\tx) -(v-v_*)\Phi_K(x)] \\
&+ 2 x\tx \int_0^K (G(z/x^\gamma)-G(z/\tx^\gamma))^2 dz \\
&+ x^2\Psi_K(x)-2(v-\tv)\cdot(v-v_*)\Psi_K(x).
\end{align*}
Recalling that $x=|v-v_*|$, $\tx=|\tv-\tv_*|$, we realize that the third line is nothing
but $A_1^K(v,v_*,\tv,\tv_*)$ while the fourth one is bounded from above by $A_3^K(v,v_*,\tv,\tv_*)$.
To conclude, it suffices to note that the sum of the terms on the two first lines equals
\begin{align*}
=&(v-v_*)\cdot[(v-v_*)-(\tv-\tv_*)-2(v-\tv)]\Phi_K(x) \\
&+ (\tv-\tv_*)\cdot[(\tv-\tv_*)-
(v-v_*)+2(v-\tv)]\Phi_K(\tx)\\
=&-(v-v_*)\cdot((v-\tv)+(v_*-\tv_*))\Phi_K(x) + (\tv-\tv_*)\cdot((v-\tv)+(v_*-\tv_*))\Phi_K(\tx)
\end{align*}
which is $A_2^K(v,v_*,\tv,\tv_*)$ as desired.
\end{proof}

Next, we study each term found in the previous inequality.
We start with the Maxwell case.

\begin{lem}\label{furthermax} Assume \eqref{cs}, \eqref{c1} with $\gamma=0$,
\eqref{c2} and adopt the notation of Lemma \ref{fundest}. For all $K\in [1,\infty)$,
all $v,v_*,\tv,\tv_* \in \rd$,

(i) $A_1^K(v,v_*,\tv,\tv_*)= 0$,

(ii) $A_2^K(v,v_*,\tv,\tv_*)=\zeta_K [-|v-\tv|^2+|v_*-\tv_*|^2]$ where $\zeta_K=\pi\int_0^K(1-\cos G(z))dz$,

(iii) $A_3^K(v,v_*,\tv,\tv_*) \leq C (|v|^2+|v_*|^2 + |\tv|^2) K^{1-2/\nu}$.
\end{lem}

\begin{proof}
Point (i) is obvious. Point (ii) immediately follows from the fact that 
$\Psi_k(x)=\zeta_K$ does not depend on $x$.
Point (iii) holds true because
$\Psi_K(x) = \pi \int_K^\infty (1-\cos G(z))dz \leq \pi \int_K^\infty G^2(z)dz \leq
C K^{1-2/\nu}$ by \eqref{eG}.
\end{proof}

The case of hard potentials is much more complicated. The following result gives 
a possible and useful upper bound on the $A^K_i$ functions. 

\begin{lem}\label{further} Assume \eqref{cs}, \eqref{c1} with $\gamma\in(0,1)$,
\eqref{c2} and adopt the notation of Lemma \ref{fundest}.

(i) For all $q>0$, there is $C_q>0$ 
such that for all $M\geq 1$, all $K\in [1,\infty)$, all $v,v_*,\tv,\tv_* \in \rd$,
\begin{align*}
A_1^K(v,v_*,\tv,\tv_*) \leq & M(|v-\tv|^2+ |v_*-\tv_*|^2) + C_q e^{-M^{q/\gamma}} e^{C_q (|v|^q+|v_*|^q)}.
\end{align*}

(ii) There is $C>0$ such that for all $K\in [1,\infty)$, all $v,v_*,\tv,\tv_* \in \rd$ and all $z_*\in\rd$,
\begin{align*}
A_2^K(v,v_*,\tv,\tv_*)- A_2^K(v,z_*,\tv,\tv_*) \leq & C\Big[|v-\tv|^2 + |v_*-\tv_*|^2 \\
& \hskip1cm + |v_*-z_*|^2(1+|v|+|v_*|+|z_*|)^{2\gamma/(1-\gamma)} \Big].
\end{align*}

(iii) There is $C>0$ such that for all $K\in [1,\infty)$, all $v,v_*,\tv,\tv_* \in \rd$,
\begin{align*}
A_3^K(v,v_*,\tv,\tv_*) \leq & C(1+ |v|^{4\gamma/\nu+2} + |v_*|^{4\gamma/\nu+2} 
+ |\tv|^2+ |\tv_*|^2) K^{1-2/\nu}.
\end{align*}
\end{lem}

This lemma is very technical. The reason is the following. The
solution $(f_t)_{t\geq 0}$ has bounded exponential moments while, on the contrary, the particle system
has only a bounded energy (moment of order $2$). If $K\in[1,\infty)$, the particle system
has all moments finite, which makes all the computations licit, but the moments of order strictly
greater than $2$ are not uniformly bounded with respect to $K$ (at least, we were not able to show it). 
We will use the previous estimates with $v,v_*$ (and $z_*$) taken from the solution $f_t$ and
$\tv,\tv_*$ taken in the particle system. Thus, it is very important that 
these estimates do not involve powers greater than $2$ of $\tv,\tv_*$. For example in point (i), 
only $v,v_*$ appear in the exponential and this is crucial.

\begin{proof} Using \eqref{c3} and that $|x^\gamma - y^\gamma|
\leq 2|x-y|/(x^{1-\gamma}+y^{1-\gamma})$, we get
\begin{align}\label{ar1}
A_1^K(v,v_*,\tv,\tv_*)\leq& 2 c_4 |v-v_*||\tv-\tv_*| \frac{(|v-v_*|^\gamma-|\tv-\tv_*|^\gamma)^2}
{|v-v_*|^\gamma+|\tv-\tv_*|^\gamma} \\
\leq & 8 c_4 \frac{|v-v_*| \land |\tv-\tv_*|}{(|v-v_*| \lor |\tv-\tv_*|)^{1-\gamma}} 
(|v-v_*|-|\tv-\tv_*|)^2.\nonumber
\end{align}
Now for any $M\geq 1$, this is bounded from above by
\begin{align*}
& \frac M 2 (|v-v_*|-|\tv-\tv_*|)^2 + 8c_4 (|v-v_*| \lor |\tv-\tv_*|)^{2+\gamma}
\indiq_{\{ 8c_4 \frac{|v-v_*| \land |\tv-\tv_*|}{(|v-v_*| \lor |\tv-\tv_*|)^{1-\gamma}} \geq \frac M2   \}}\\
\leq & \frac M2 (|v-\tv|+|v_*-\tv_*|)^2 + 8 c_4 
\left[\frac{16c_4}{M}(|v-v_*| \land |\tv-\tv_*|)\right]^{\frac{2+\gamma}{1-\gamma}}
\indiq_{\{ \frac{|v-v_*| \land |\tv-\tv_*|}{(|v-v_*| \lor |\tv-\tv_*|)^{1-\gamma}} \geq \frac M {16 c_4}   \}} \\
\leq &  M(|v-\tv|^2+|v_*-\tv_*|^2) + 8 c_4 
\left[16c_4 (|v-v_*| \land |\tv-\tv_*|)\right]^{\frac{2+\gamma}{1-\gamma}}
\indiq_{\{ (|v-v_*| \land |\tv-\tv_*|)^\gamma \geq \frac M {16c_4}  \}}\\
\leq &  M(|v-\tv|^2+|v_*-\tv_*|^2) + 8 c_4 
\left[16c_4 (|v|+|v_*|)\right]^{\frac{2+\gamma}{1-\gamma}}
\indiq_{\{ (|v|+|v_*|)^\gamma \geq \frac M {16c_4}  \}}
\end{align*}
Fix now $q>0$ and observe that 
$$
x^{\frac{2+\gamma}{1-\gamma}} \indiq_{\{ x^\gamma \geq  \frac M { 16c_4} \}} \leq 
x^{\frac{2+\gamma}{1-\gamma}} e^{-M^{q/\gamma}} e^{(16c_4)^{q/\gamma} x^q } \leq C_q  e^{-M^{q/\gamma}} e^{2(16c_4)^{q/\gamma} x^q }.
$$
Point (i) follows.

\vip

Point (ii) is quite delicate. First, there is $C$ such that for all $K\in [1,\infty)$, all $x,y>0$,
$$
\Phi_K(x) \leq C x^\gamma \quad \hbox{and}\quad |\Phi_K(x)-\Phi_K(y)| \leq C |x^\gamma-y^\gamma|.
$$
Indeed, it is enough to prove that for $\Gamma_K(x)=\int_0^K(1-\cos G(z/x))dz$, 
$\Gamma_K(0)=0$ and  $|\Gamma_K'(x)| \leq C$. But $\Gamma_K(x)=x \int_0^{K/x}(1-\cos G(z))dz 
\le x \int_0^\infty G^2(z)dz  $,
so that $\Gamma_K(0)=0$ and $|\Gamma_K'(x)| \leq \int_0^\infty (1-\cos G(z))dz + x(K/x^2) (1-\cos G(K/x))
\leq \int_0^\infty G^2(z)dz + (K/x) G^2(K/x)$, which is uniformly bounded by \eqref{eG}.
Consequently, for all $X,Y\in\rd$, 
\begin{align*}
|X \Phi_K(|X|)- Y \Phi_K(|Y|)| \leq& C |X-Y| (|X|^\gamma + |Y|^\gamma) + C (|X|+|Y|) | |X|^\gamma - |Y|^\gamma|.
\end{align*}
Using again that $|x^\gamma - y^\gamma|\leq 2 |x-y|/(x^{1-\gamma}+y^{1-\gamma})$, we easily conclude that
\begin{align}\label{fi}
|X \Phi_K(|X|)- Y \Phi_K(|Y|)| \leq& C |X-Y| (|X|^\gamma + |Y|^\gamma).
\end{align}
Now we write
\begin{align}\label{topcool1}
\Delta_2^K:=&A_2^K(v,v_*,\tv,\tv_*)-A_2^K(v,z_*,\tv,\tv_*) \nonumber \\
=& -  \big[(v-\tv)+(v_*-\tv_*)\big]\cdot\big[(v-v_*) \Phi_K(|v-v_*|)-(\tv-\tv_*)\Phi_K(|\tv-\tv_*|)\big]
\nonumber\\
&+\big[(v-\tv)+(z_*-\tv_*)\big]\cdot\big[(v-z_*) \Phi_K(|v-z_*|)-(\tv-\tv_*)\Phi_K(|\tv-\tv_*|)\big] 
\nonumber\\
=& -  \big[(v-\tv)+(v_*-\tv_*)\big]\cdot\big[(v-v_*) \Phi_K(|v-v_*|)-(v-z_*) \Phi_K(|v-z_*|)\big]\\
&+ (z_*-v_* )\cdot\big[(v-z_*) \Phi_K(|v-z_*|)-(\tv-\tv_*)\Phi_K(|\tv-\tv_*|)\big].\nonumber
\end{align}
By \eqref{fi} and the Young inequality, we deduce that
\begin{align*}
\Delta_2^K\leq &C (|v-\tv|+|v_*-\tv_*|) |v_*-z_*| (|v-v_*|^\gamma+ |v-z_*|^\gamma)\\
&+C |z_*-v_* | (|v-\tv|+|z_*-\tv_*|)(|v-z_*|^\gamma + |\tv-\tv_*|^\gamma)\\
\leq & C  [(|v-\tv|+|v_*-\tv_*|)^2 +  |v_*-z_*|^2 (|v-v_*|^\gamma+ |v-z_*|^\gamma)^2 ]\\
&+ C|z_*-v_* | (|v-\tv|+|z_*-v_*|+|v_*-\tv_*|)(|v-z_*|^\gamma + (|v-\tv|+|v-v_*|+|v_*-\tv_*|)^\gamma).
\end{align*}
The first term is clearly bounded by
$C  (|v-\tv|^2+|v_*-\tv_*|^2 + |v_*-z_*|^2 (1+|v| +|v_*|+|z_*|)^{2\gamma})$ which fits the statement, since
$2\gamma \leq 2\gamma/(1-\gamma)$. We next bound the second term by
\begin{align*}
&C|z_*-v_* |^2 (|v-z_*| + |v-v_*|)^\gamma\\
&+ C|z_*-v_* |^2 (|v-\tv|+|v_*-\tv_*|)^\gamma \\
&+ C |z_*-v_* |(|v-\tv|+|v_*-\tv_*|) (|v-z_*|+|v-v_*|)^\gamma\\
&+ C  |z_*-v_* |(|v-\tv|+|v_*-\tv_*|)^{1+\gamma}.
\end{align*}
Using that $x^2 y^\gamma \leq  x^{4/(2-\gamma)} + y^2$ (for the second line), that $xyz^\gamma\leq (xz^\gamma)^2
+y^2$ (for the third line) and that $xy^{1+\gamma} \leq x^{2/(1-\gamma)} + y^2$, we obtain the upper-bound
\begin{align*}
&C|z_*-v_* |^2 (1+|v|+|z_*| + |v_*|)^\gamma\\
&+ C (|v-\tv|+|v_*-\tv_*|)^2 + |z_*-v_* |^{4/(2-\gamma)} \\
&+ C (|v-\tv|+|v_*-\tv_*|)^2 + |z_*-v_* |^2(|v-z_*|+|v-v_*|)^{2\gamma}\\
&+ C  (|v-\tv|+|v_*-\tv_*|)^2+ |z_*-v_*|^{2/(1-\gamma)},
\end{align*}
which is bounded by
\begin{align*} 
C (|v-\tv|^2+|v_*-\tv_*|^2) + C |z_*-v_*|^2\Big\{(1+|v|+|z_*| + |v_*|)^\gamma
+|z_*-v_* |^{4/(2-\gamma)-2} \\
+ (|v-z_*|+|v-v_*|)^{2\gamma}+|z_*-v_*|^{2/(1-\gamma)-2}  \Big\}.
\end{align*}
One easily concludes, using that $\max\{\gamma,4/(2-\gamma)-2,2\gamma, 2/(1-\gamma)-2\}=
2\gamma/(1-\gamma)$.

\vip

We finally check point (iii). Using \eqref{eG}, we deduce that $1-\cos(G(z/x^\gamma))\leq G^2(z/x^\gamma)
\leq C (z/x^\gamma)^{-2/\nu}$, whence $\Psi_K(x) \leq C x^{2\gamma/\nu} \int_K^\infty z^{-2/\nu}dz
= Cx^{2\gamma/\nu} K^{1-2/\nu}$. Thus
\begin{align}\label{ar2}
A_3^K(v,v_*,\tv,\tv_*) \leq & C(|v-v_*|^2+|v-v_*||\tv-\tv_*|)|v-v_*|^{2\gamma/\nu} K^{1-2/\nu},
\end{align}
from which we easily conclude, using that $|\tv-\tv_*||v-v_*|^{1+2\gamma/\nu}\leq |\tv-\tv_*|^2+
|v-v_*|^{2+4\gamma/\nu}$.
\end{proof}

We conclude with the hard spheres case.

\begin{lem}\label{furtherhs} Assume \eqref{cs}, \eqref{c1} with $\gamma=1$,
\eqref{c2hs} and adopt the notation of Lemma \ref{fundest}.

(i) For all $q>0$, there is $C_q>0$ 
such that for all $M\geq 1$, all $K\in [1,\infty)$, all $v,v_*,\tv,\tv_* \in \rd$,
\begin{align*}
A_1^K(v,v_*,\tv,\tv_*) \leq & M(|v-\tv|^2+ |v_*-\tv_*|^2) + C_q K(|\tv|+|\tv_*| )
e^{-M^{q}} e^{C_q (|v|^q+|v_*|^q)}.
\end{align*}

(ii) For all $q>0$, there is $C_q>0$ 
such that for all $M\geq 1$, all $K\in [1,\infty)$, all $v,v_*,\tv,\tv_* \in \rd$,
\begin{align*}
A_2^K(v,v_*,\tv,\tv_*)- A_2^K(v,z_*,\tv,\tv_*) \leq & M(|v-\tv|^2 + |v_*-\tv_*|^2)
+  C |v_*-z_*|^2(1+|v|+|v_*|+|z_*|)^{2}\\
&+ C_q (1+|\tv|+|\tv_*|) Ke^{- M^q } e^{ C_q(|v|^q+|v_*|^q+|z_*|^q)}
\end{align*}

(iii) For all $q>0$, there is $C_q>0$ such that for all $K\in [1,\infty)$, all $v,v_*,\tv,\tv_* \in \rd$,
\begin{align*}
A_3^K(v,v_*,\tv,\tv_*) \leq & C_q(1+|\tv|)e^{-K^q}e^{ C_q(|v|^q+|v_*|^q+|z_*|^q)}.
\end{align*}
\end{lem}

\begin{proof}
On the one hand, \eqref{c3} implies
\begin{align*}
A_1^K(v,v_*,\tv,\tv_*)\leq& 2c_4 |v-v_*||\tv-\tv_*|\frac{(|v-v_*|-|\tv-\tv_*|)^2}{ |v-v_*|+ |\tv-\tv_*| }\\
\leq& 4c_4 (|v-v_*|\land|\tv-\tv_*|) (|v-\tv|^2+ |v_*-\tv_*|^2).
\end{align*}
On the other hand, since $G$ takes values in $(0,\pi/2)$, we obviously have
\begin{align*}
A_1^K(v,v_*,\tv,\tv_*)\leq& \frac{\pi^2}2 K |v-v_*||\tv-\tv_*|.
\end{align*}
Consequently, we may write
\begin{align*}
A_1^K(v,v_*,\tv,\tv_*)\leq M (|v-\tv|^2+ |v_*-\tv_*|^2) + \frac{\pi^2}2 K |v-v_*||\tv-\tv_*|
\indiq_{\{ 4c_4 (|v-v_*|\land|\tv-\tv_*|) \geq M  \}}.
\end{align*}
Point (i) easily follows, using that $|v-v_*|\indiq_{\{ 4c_4 (|v-v_*|\land|\tv-\tv_*|) \geq M  \}}
\leq |v-v_*|\indiq_{\{ 4c_4 |v-v_*| \geq M  \}} \leq |v-v_*| e^{-M^q}e^{ (4c_4 |v-v_*|)^q} \leq
C_q e^{-M^q}e^{ 2(4c_4 |v-v_*|)^q}\leq C_q e^{-M^q}e^{ 2^{q+1}(4c_4)^q (|v|^q+|v_*|^q)}$.

\vip

Using all the computations of the proof of Lemma \ref{further}-(ii) except the one that makes
appear the power $2/(1-\gamma)$, we see that
for $\Delta_2^K:=A_2^K(v,v_*,\tv,\tv_*)-A_2^K(v,z_*,\tv,\tv_*)$
\begin{align*}
\Delta_2^K
\leq& C [|v-\tv|^2 + |v_*-\tv_*|^2 + |v_*-z_*|^2(1+|v|+|v_*|+|z_*|)^2  +  |z_*-v_*|(|v-\tv|^2 + |v_*-\tv_*|^2)]
\\
\leq & C(1+|z_*-v_*|)(|v-\tv|^2 + |v_*-\tv_*|^2) +  C|v_*-z_*|^2(1+|v|+|v_*|+|z_*|)^2 .
\end{align*}
On the other hand, starting from \eqref{topcool1} and using that $\phi_K(x)\leq \pi K$, we realize that
\begin{align*}
\Delta_2^K \leq C K (1+|\tv|+|\tv_*|) (1+|v|^2+|v_*|^2+|z_*|^2).
\end{align*}
Hence we can write, for any $M>1$,
\begin{align*}
\Delta_2^K
\leq & M (|v-\tv|^2 + |v_*-\tv_*|^2) +  C |v_*-z_*|^2(1+|v|+|v_*|+|z_*|)^2 \\
&+C K (1+|\tv|+|\tv_*|) (1+|v|^2+|v_*|^2+|z_*|^2)\indiq_{\{ C(1+|z_*-v_*|) \geq M\}}.
\end{align*}
But $(1+|v|^2+|v_*|^2+|z_*|^2)\indiq_{\{ C(1+|z_*-v_*|) \geq M\}}
\leq (1+|v|+|v_*|+|z_*|)^2\indiq_{\{ C(1+|v|+|v_*|+|z_*|) \geq M\}} \leq 
(1+|v|+|v_*|+|z_*|)^2 e^{- M^q } e^{ C^q(1+|v|+|v_*|+|z_*|)^q } \leq C_q e^{- M^q } e^{ C_q(|v|^q+|v_*|^q+|z_*|^q)}$.
Point (ii) is checked.

\vip

Finally, we observe that $\Psi_K(x)\leq \pi \int_K^\infty G^2(z/x) dz$. But here,
$G(z)=(\pi/2-z)_+$ whence
$\Psi_K(x)\leq (\pi^4/24)x \indiq_{\{x \geq 2K/\pi\}} \leq 5x \indiq_{\{x \geq K/2\}}$.
Thus for any $q>0$, $\Psi_K(x) \leq 5x e^{-K^q}e^{2^q x^q}$, so that
\begin{align*}
A^K_3(v,v_*,\tv,\tv_*)\leq& C (1+|\tv|)(1+|v|^2+|v_*|^2)e^{-K^q}|v-v_*|e^{2^q |v-v_*|^q}\\
\leq& C_q (1+|\tv|) e^{-K^q} e^{C_q (|v|^q+|v_*|^q)}
\end{align*}
as desired.
\end{proof}

\section{Convergence of the particle system with cutoff}
\setcounter{equation}{0} 
\label{sec:ConvCutoff}

To build a suitable coupling between the particle system and the solution to \eqref{be},
we need to introduce the (stochastic) paths associated to \eqref{be}.
To do so, we follow the ideas of Tanaka \cite{T1978,T1979} and make use of two probability spaces.
The main one is an abstract $(\Omega,\cF,\Pr)$, on which the random objects are defined 
when nothing is precised. But we will also need an auxiliary one, 
$[0,1]$ endowed with its Borel $\sigma$-field 
and its Lebesgue measure. 
In order to avoid confusion, a random variable defined on this latter probability space will be called an
$\alpha$-random variable, expectation on $[0,1]$ will be denoted by $\E_\alpha$, etc.

\subsection{A SDE for the Boltzmann equation}

First, we recall the classical probabilistic interpretation of the Boltzmann equation
initiated by Tanaka \cite{T1978,T1979} in the Maxwell molecules case. 

\begin{prop}\label{wt}
Assume \eqref{cs}, \eqref{c1}, \eqref{c2hs} or \eqref{c2} and let $f_0 \in \cP_2(\rd)$.
If $\gamma\in (0,1]$, assume additionally \eqref{c4} and
that $f_0$ satisfies \eqref{c5}.
Let $(f_t)_{t\geq 0}$ be the corresponding unique weak solution to \eqref{be}.
Consider any $f_0$-distributed random variable $W_0$ and any independent Poisson measure
$M(ds,d\alpha,dz,d\varphi)$ on $[0,\infty)\times [0,1]\times[0,\infty)\times [0,2\pi)$ with intensity measure
$ds d\alpha dz d\varphi$. Consider also, for each $t\geq 0$, a $f_t$-distributed $\alpha$-random 
variable $W_t^*$, in such a way that $(t,\alpha)\mapsto W^*_t(\alpha)$ is measurable.
Then there is a unique (c\`adl\`ag adapted) strong solution to
\begin{align}\label{sde}
W_t=W_0 + \intot\int_\rd\int_0^\infty\int_0^{2\pi} c(W_\sm,W_s^*(\alpha),z,\varphi) M(ds,d\alpha,dz,d\varphi).
\end{align}
Furthermore, $W_t$ is $f_t$-distributed for each $t\geq 0$.
\end{prop}

We will note $(W_t)_{t\geq 0}$ such a Boltzmann process. It can be viewed as the time-evolution of the velocity
of a typical particle in the gas.

\begin{proof}
The proof is very similar to that of \cite[Proposition 5.1]{F2012}, see also
\cite[Section 4]{FGu2008} and is omitted. In \cite[Proposition 5.1]{F2012}, the same Boltzmann 
equation is studied, with much less assumptions on $f_0$ (so that uniqueness is not known for \eqref{be}).
But the formulation of the SDE is different (it is equivalent in law). 
The same proof as in \cite[Proposition 5.1]{F2012} works here,
with several difficulties avoided due to the facts that $f_0$ has exponential moments and
that uniqueness is known to hold for \eqref{be}. 
\end{proof}

\subsection{A SDE for the particle system}

Here we write down a Poisson stochastic differential equation corresponding to Nanbu's particle system
and we prove Proposition \ref{wps}-(i).

\begin{prop}\label{wpst}
Assume \eqref{cs}, \eqref{c1}, \eqref{c2hs} or \eqref{c2} 
and let $f_0 \in \cP_2(\rd)$, $N\geq 1$ and $K\in [1,\infty)$.
Consider a family $(V^i_0)_{i=1,\dots N}$ of i.i.d. $f_0$-distributed random variables and an independent
family  $(O^N_i(ds,dj,dz,d\varphi))_{i=1,\dots,N}$ of 
Poisson measures on
$[0,\infty)\times \{1,\dots,N\} \times[0,\infty)\times [0,2\pi)$ with intensity measures
$ds \left(N^{-1}\sum_{k=1}^N \delta_k(dj)\right) dz d\varphi$.
There exists a unique (c\`adl\`ag and adapted) strong solution to
\begin{align}\label{pssde}
V^{i,N,K}_t=V^i_0 + \intot\int_j\int_0^\infty\int_0^{2\pi} c_K(V^{i,N,K}_\sm,V^{j,N,K}_\sm,z,\varphi) 
O^N_i(ds,dj,dz,d\varphi),
\quad i=1,\dots,N.
\end{align}
Furthermore, $(V^{i,N,K}_t)_{i=1,\dots,N, t\geq 0}$ is Markov with generator 
$\cL_{N,K}$. We have $\E\left[|V^{1,N,K}_t|^2\right]= \intrd|v|^2f_0(dv)$ and, if 
$ \intrd|v|^pf_0(dv)$ for some $p\geq 2$, $\sup_{[0,T]}\E\left[|V^{1,N,K}_t|^p\right] \leq C_{p,T,f_0,K}$.
\end{prop}

\begin{proof}
First of all, observe that we actually deal with finite Poisson measures, since $c_K$ vanishes
for $z\geq K$. Thus, strong existence and uniqueness for \eqref{pssde} is trivial: it suffices
to work recursively on the instants of jumps (which are discrete) of the family 
$(O^N_i(ds,dj,dz,d\varphi))_{i=1,\dots,N}$. Consequently, $\bV_t^{N,K}=(V^{1,N,K}_t,\dots,V^{N,N,K}_t)$
is a Markov process, since it solves a well-posed time-homogeneous SDE. Its infinitesimal generator
is classically defined by \eqref{lKgood}, with actually a sum over all couples $(i,j)\in\{1,\dots,N\}^2$,
but this changes nothing since the terms with $i=j$ vanish because
$c_K(v,v,z,\varphi)=0$ for all $v\in\rd$. Next, a simple computation shows that
\begin{align*}
\E[|V^{1,N,K}_t|^2]=&\E[|V^1_0|^2] + \frac 1 N \sum_{j=1}^N \intot\int_0^\infty\int_0^{2\pi} 
\E\Big(|V^{1,N,K}_s+c_K(V^{1,N,K}_s,V^{j,N,K}_s,z,\varphi)|^2 \\
&\hskip8cm- |V^{1,N,K}_s|^2\Big) d\varphi dz ds\\
=&\E[|V^1_0|^2] + \frac{N-1}{N} \intot\int_0^K \int_0^{2\pi} 
\E\Big(|c(V^{1,N,K}_s,V^{2,N,K}_s,z,\varphi)|^2\\
&\hskip4cm +2 V^{1,N,K}_s\cdot c(V^{1,N,K}_s,V^{2,N,K}_s,z,\varphi)  
\Big) d\varphi dz ds
\end{align*}
by exchangeability. But, as seen in the proof of Lemma \ref{fundest},
$$
\int_0^K \int_0^{2\pi} \Big(|c(v,v_*,z,\varphi)|^2 + 2v\cdot c(v,v_*,z,\varphi)\Big) d\varphi dz =
[|v-v_*|^2 -2v\cdot (v-v_*)]\Phi_K(|v-v_*|),
$$
whence, using again exchangeability,
\begin{align*}
\E[|V^{1,N,K}_t|^2]= &\E[|V^1_0|^2] + \frac{N-1}{N} \intot
\E\Big(\Big[|V^{1,N,K}_s-V^{2,N,K}_s|^2 -2V^{1,N,K}_s\cdot (V^{1,N,K}_s-V^{2,N,K}_s)\Big]\\
&\hskip7cm \Phi_K(|V^{1,N,K}_s-V^{2,N,K}_s|)\Big) ds\\
=&\E[|V^1_0|^2] + \frac{N-1}{N} \intot
\E\Big(\Big[|V^{1,N,K}_s-V^{2,N,K}_s|^2 -V^{1,N,K}_s\cdot (V^{1,N,K}_s-V^{2,N,K}_s)\\
&\hskip3cm  -V^{2,N,K}_s\cdot (V^{2,N,K}_s-V^{1,N,K}_s) \Big]  \Phi_K(|V^{1,N,K}_s-V^{2,N,K}_s|)\Big) ds.
\end{align*}
In this last expression, the integrand is zero, so that, as claimed, $\E[|V^{1,N,K}_t|^2]= \E[|V^1_0|^2] 
=\intrd |v|^2f_0(dv)$. Recalling finally \eqref{dfvprime} and \eqref{dfc}, we see that
$|c(v,v_*,z,\varphi)|\leq |v-v_*|$. Thus for $p\geq 2$,
\begin{align*}
\int_0^K \int_0^{2\pi} (|v+c(v,v_*,z,\varphi)|^p-|v|^p) d\varphi dz \leq C_{p} K (|v|+|v_*|^p).
\end{align*}
Consequently, we obtain as previously
$$
\E[|V^{1,N,K}_t|^p]\leq \E[|V^1_0|^p] +  \frac {C_pK} N \sum_{j=1}^N \intot \E[|V^{1,N,K}_s|^p+|V^{j,N,K}_s|^p ]ds
$$
and conclude, using again exchangeability, that $\E[|V^{1,N,K}_t|^p]\leq \E[|V^1_0|^p]e^{2C_pKt}$ as desired.
\end{proof}

This allows us to deduce

\begin{preuve} {\it of Proposition \ref{wps}-(i).} The strong existence and uniqueness
for the SDE \eqref{pssde} 
classically implies the existence and uniqueness of a Markov
process with generator $\cL_{N,K}$.
\end{preuve}

\subsection{The coupling}

Here we explain how we couple our particle system with a family of i.i.d. Boltzmann processes. 
For example, we want to couple $V^{1,N,K}_t$ with a Boltzmann process $W^1_t$. 
The main difficulty is that at each collision, 
$W^1_t$ is collided by an independent particle (using $W^*_t$) while
$V^{1,N,K}_t$ is collided by some $V^{j,N,K}_t$. We thus have to choose $j$ in such a way that
$V^{j,N,K}_t$ is as close as possible to $W^*_t$, but $j$ has to remain uniformly chosen.

\vip

A technical problem obliges us to introduce the set $(\rd)^N_\bullet := \{\bw\in(\rd)^N : \; 
w_i\ne w_j \;\forall \; i \ne j\}$.

\begin{lem}\label{super}
Let $f_t \in C([0,\infty),\cP_2(\rd))$ be such that $f_t$ has a density for all $t>0$. Let
also $N\geq 1$ be fixed. 
For $\bv=(v_1,\dots,v_N) \in (\rd)^N$, we denote by $\mu^N_{\bv}:=N^{-1}\sum_1^N \delta_{v_i}$ the empirical 
measure associated to $\bv$. There exists a measurable map 
$(t,\bw,\bv,\alpha)\mapsto (W_t^*(\alpha),Z_t^*(\bw,\alpha),V_t^*(\bv,\bw,\alpha))$ from
$(0,\infty)\times (\rd)^N_\bullet \times (\rd)^N\times[0,1]$ into $\rd\times\rd\times\rd$ enjoying the following
properties

(a) for all $t\geq 0$, the $\alpha$-law of $W^*_t$ is $f_t$,

(b) for all $t\geq 0$, $\bw\in (\rd)^N_\bullet$, the $\alpha$-law of $Z_t^*(\bw,.)$ is $\mu^N_\bw$,

(c) for all $t\geq 0$, $\bw\in (\rd)^N_\bullet$, $\bv\in(\rd)^N$, the $\alpha$-law of
$V_t^*(\bv,\bw,.)$ is $\mu^N_\bv$,

(d) for all $t\geq 0$, $\bw\in (\rd)^N_\bullet$, $\bv\in (\rd)^N$,  the $\alpha$-law of
$(Z^*_t(\bw,.),V_t^*(\bv,\bw,.))$
is $N^{-1}\sum_1^N \delta_{(w_i,v_i)}$,

(e) for all $t\geq 0$, all $\bw\in (\rd)^N_\bullet$, 
$\int_0^1 |W_t^*(\alpha)- Z^*_t(\bw,\alpha) |^2 d\alpha = \cW_2^2(f_t, \mu^N_\bw)$.
\end{lem}

\begin{proof}
We first consider, for each $t> 0$, $W_t^*$ such that point (a) holds true and such that
$(t,\alpha)\mapsto W_t^*(\alpha)$ is measurable.

Next, we recall that by Brenier's theorem (see e.g. Villani \cite[Theorem 2.12 p 66]{Vtot2003})
for each $t>0$ and each $\bw\in(\rd)^N$, since $f_t$ does does not
charge small sets (because it has a density by \cite{F2012}), there 
exists a unique map $F_{t,\bw}:\rd\mapsto\rd$
such that, setting $Z_t^*(\bw,\alpha):=F_{t,\bw}(W_t^*(\alpha))$, points (b) and (e) hold true.
In other words, $(W_t^*(.),Z_t^*(\bw,.))$ is an optimal coupling for $f_t$ and $\mu^N_{\bw}$.
Furthermore, Fontbona-Gu\'erin-M\'el\'eard \cite{FGM} have shown that $F_{t,\bw}(x)$ is a measurable
function of $(t,\bw,x)$. Consequently, $Z_t^*(\bw,\alpha)$ is a measurable function of $(t,\bw,\alpha)$.

Finally, we define, for any $\bw \in (\rd)^N_\bullet$ and any $\bv\in(\rd)^N$, 
the map $G_{\bw,\bv}:\{w_1,\dots,w_N\}\mapsto \{v_1,\dots,v_N\}$ by $G_{\bw,\bv}(w_i)=v_i$
(here we need that $\bw \in (\rd)^N_\bullet$). We then we put 
$V_t^*(\bv,\bw,\alpha)=G_{\bw,\bv}(Z_t^*(\bw,\alpha))$,
which is clearly measurable (in all its variables). Point (d) follows from (b) and the definition
of $G_{\bw,\bv}$ and finally (c) follows from (d).
\end{proof}

Here is the coupling we propose.

\begin{lem}\label{coupling}
Assume \eqref{cs}, \eqref{c1}, \eqref{c2hs} or \eqref{c2}. Let $f_0 \in \cP_2(\rd)$.
Assume additionally \eqref{c4} and \eqref{c5} if $\gamma\in (0,1]$.
Let $(f_t)_{t\geq 0}$ be the unique weak solution to \eqref{be} and assume
that $f_t$ has a density for all $t>0$ (see Theorem \ref{wp}).
Consider $N\geq 1$ and $K\in [1,\infty)$ fixed. Let $(V_0^i)_{i=1,\dots N}$
be i.i.d. with common law $f_0$ and let $(M_i(ds,d\alpha,dz,d\varphi))_{i=1,\dots,N}$ be an i.i.d. family of 
Poisson measures on $[0,\infty)\times [0,1]\times[0,\infty)\times [0,2\pi)$ with intensity measures
$ds d\alpha dz d\varphi$, independent of $(V_0^i)_{i=1,\dots N}$.

(i) The following SDE's, for $i=1,\dots,N$, define $N$ independent copies of the Boltzmann process:
\begin{align*}
W^{i}_t=&V^i_0 + \intot\int_0^1\int_0^\infty\int_0^{2\pi} 
c(W^{i}_\sm,W^*_s(\alpha),z,\varphi) M_i(ds,d\alpha,dz,d\varphi).
\end{align*}
In particular, for each $t\geq 0$, $(W^i_t)_{i=1,\dots,N}$ are i.i.d. with common law $f_t$. Consequently,
since $f_t$ has a density for all $t>0$, $(W^i_t)_{i=1,\dots,N}\in(\rd)^N_\bullet$ a.s.

(ii) Next, we consider the system of SDE's, for $i=1,\dots,N$,
\begin{align*}
V^{i,N,K}_t=V^i_0 + \intot\int_0^1\int_0^\infty\int_0^{2\pi} 
c_K(V^{i,N,K}_\sm,V^*_s(\bV^{N,K}_\sm,\bW_\sm,\alpha),z,\varphi+\varphi_{i,\alpha,s}) M_i(ds,d\alpha,dz,d\varphi),
\end{align*}
where we used the notation $\bV^{N,K}_\sm=(V^{1,N,K}_\sm,\dots,V^{N,N,K}_\sm)\in(\rd)^N$,
$\bW_\sm=(W^1_\sm,\dots,W^N_\sm)\in(\rd)^N_\bullet$ and where we have set
$\varphi_{i,\alpha,s}:=\varphi_0(W^{i}_\sm-W^*_s(s,\alpha),V^{i,N,K}_\sm-V^*_s(\bV^{N,K}_\sm,\bW_\sm,\alpha))$
for simplicity. This system of SDEs has a unique solution, and this solution
is a Markov process with generator $\cL_{N,K}$ and initial condition $(V_0^i)_{i=1,\dots N}$.

(iii) The family $((W^1_t,V^{1,N,K}_t)_{t\geq 0},...,(W^N_t,V^{N,N,K}_t)_{t\geq 0})$ is exchangeable.
\end{lem}

\begin{proof}
Point (i) is a direct consequence of Proposition \ref{wt} and point (iii) follows from the exchangeability
of the family $(V^i_0,M_i)_{i=1,\dots,N}$ and from uniqueness (in law).
In point (ii), the existence and uniqueness result is also immediate, since the Poisson
measures under consideration are finite (or rather, are finite when $z$ is restricted to $[0,K]$,
which is the case since $c_K=c\indiq_{\{z\leq K\}}$). Finally
$(V^{1,N,K}_t,\dots,V^{N,N,K}_t)_{t\geq 0}$ is a Markov process with generator $\cL_{N,K}$ due to the
fact that for all $\bv\in(\rd)^N$, all $\bw \in(\rd)^N_\bullet$, all $s> 0$, all 
$\varphi_{ij} \in [0,2\pi)$, for all bounded measurable function $\phi:(\rd)^N\mapsto \rr$,
\begin{align*}
&\sum_{i=1}^N
\int_0^1 \int_0^\infty \int_0^{2\pi} \Big(\phi(\bv + c_K(v_i,V_s^*(\bv,\bw,\alpha),z,\varphi+\varphi_{ij}).\be_i)
-\phi(\bv) \Big)  d\varphi dz d\alpha\\
=&\sum_{i=1}^N \frac 1N \sum_{j=1}^N 
\int_0^\infty \int_0^{2\pi} \Big(\phi(\bv + c_K(v_i,v_j,z,\varphi+\varphi_{ij}).\be_i)
-\phi(\bv) \Big)  d\varphi dz\\
=&\frac 1N \sum_{i\ne j}
\int_0^\infty \int_0^{2\pi} \Big(\phi(\bv + c_K(v_i,v_j,z,\varphi).\be_i)
-\phi(\bv) \Big)  d\varphi dz,
\end{align*}
which is nothing but $\cL_{N,K}\phi(\bv)$, see \eqref{lKgood}. We used Lemma \ref{super}-(c) for the
first equality and the $2\pi$-periodicity of $c_K$ (in $\varphi$) and the fact that 
$c_K(v_i,v_i,z,\varphi)=0$ for the second one. 
\end{proof}

\subsection{Estimate of the Wasserstein distance}

We can now prove our main result in the case with cutoff. We first study hard potentials.

\begin{preuve} {\it of Theorem \ref{mr}-(ii) when $K\in [1,\infty)$.} 
We thus assume \eqref{cs}, \eqref{c1} with $\gamma\in(0,1)$ and \eqref{c2}. We consider $f_0\in \cP_2(\rd)$
satisfying \eqref{c5} for some $p\in(\gamma,2)$ and fix $q\in(\gamma,p)$ for the rest of the proof.
We also assume that $f_0$ is not a Dirac mass, so that $f_t$ has a density for all $t>0$.
We fix $N\geq 1$ and $K\in [1,\infty)$ and consider 
the processes introduced in Lemma \ref{coupling}.

\vip

{\it Step 1.} A direct application of the It\^o calculus for jump processes
shows that
\begin{align*}
&\E[|W^{1}_t -V^{1,N,K}_t |^2]\\
=& \intot \int_0^1 \int_0^\infty \int_0^{2\pi} \E\Big[ |W^{1}_s -V^{1,N,K}_s + \Delta^1(s,\alpha,z,\varphi) |^2 - 
 |W^{1}_s -V^{1,N,K}_s|^2 \Big] d\varphi dz d\alpha ds,
\end{align*}
where
\begin{align*}
\Delta^1(s,\alpha,z,\varphi)= c(W^{1}_s,W^*_s(s,\alpha),z,\varphi)-
c_K(V^{1,N,K}_s,V^*_s(\bV^{N,K}_s,\bW_s,\alpha),z,\varphi+\varphi_{i,\alpha,s}).
\end{align*}
Using Lemma \ref{fundest}, we thus obtain
\begin{align*}
\E[|W^{1}_t -V^{1,N,K}_t |^2] \leq \intot [B_1^K(s)+B_2^K(s)+B_3^K(s) ]ds,
\end{align*}
where, for $i=1,2,3$,
\begin{align*}
B_i^K(s):= \int_0^1 \E\Big[A_i^K(W^{1}_s, W^*_s(\alpha), V^{1,N,K}_s, V^*_s(\bV^{N,K}_s,\bW_s,\alpha)) \Big] 
d\alpha.
\end{align*}

\vip

{\it Step 2.} Using Lemma \ref{further}-(i),
we see that for all $M\geq 1$ (recall that $q\in(\gamma,p)$
is fixed).
\begin{align*}
B_1^K(s) \leq& M \int_0^1 \E\left[ |W^{1}_s- V^{1,N,K}_s |^2+ |W^*_s(\alpha)-
V^*_s(\bV^{N,K}_s,\bW_s,\alpha)) |^2 \right] d\alpha \\
&+ C e^{-M^{q/\gamma}} \int_0^1 \E \left[ \exp ( C (|W^{1}_s|^q+|W_s^*(\alpha)|^q)) \right] d\alpha\\
\leq &  M \int_0^1 \E\left[ |W^{1}_s- V^{1,N,K}_s |^2+ |W^*_s(\alpha)-
V^*_s(\bV^{N,K}_s,\bW_s,\alpha)) |^2 \right] d\alpha + C e^{-M^{q/\gamma}}.
\end{align*}
To get the last inequality, we used that $W^1_s$ and $W^*_s(.)$ are independent and satisfy 
$W^1_s \sim f_s$ and 
$W^*_s(.)\sim f_s$, whence
\bean
\int_0^1 \E \left[ \exp ( C (|W^{1}_s|^q +|W_s^*(\alpha)|^q) \right] d\alpha 
&=& \Bigl( \int_{\rd} e^{ C|w|^q} f_s(dw) \Bigr)^2<\infty
\eean
by \eqref{momex}.

\vip

{\it Step 3.} Roughly speaking, $B^K_2$ should not be far to be zero for symmetry reasons. We claim
  that $B^K_2$ would be zero if $W^*_s(\alpha)$ was replaced by $Z^*_s(\bW_s,\alpha)$. 
More precisely, we check here that
$$
\tB^K_2(s):=\int_0^1\E\left[A_2^K(W^{1}_s, Z^*_s(\bW_s,\alpha), 
V^{1,N,K}_s, V^*_s(\bV^{N,K}_s,\bW_s,\alpha)) \right] d\alpha=0.
$$
By Lemma \ref{super}-(d), we simply have
\begin{align*}
\tB^K_2(s)=&\E\left[ \frac 1 N \sum_{i=1}^N A_2^K(W^{1}_s, W^i_s,V^{1,N,K}_s,V^{i,N,K}_s) \right]
= \frac{N-1}N\E\left[A_2^K(W^{1}_s, W^2_s,V^{1,N,K}_s,V^{2,N,K}_s) \right]
\end{align*}
by exchangeability and since $A_2^K(v,v,\tv,\tv)=0$.
Finally, we write, using again exchangeability,
\begin{align*}
\tB^K_2(s)=& \frac{N-1}{2N}\E\left[A_2^K(W^{1}_s, W^2_s,V^{1,N,K}_s,V^{2,N,K}_s)
+ A_2^K(W^{2}_s, W^1_s,V^{2,N,K}_s,V^{1,N,K}_s) \right].
\end{align*}
This is zero by symmetry of $A_2^K$: it holds that $A_2^K(v,v_*,\tv,\tv_*)+A_2^K(\tv,\tv_*,v,v_*)=0$.

\vip

{\it Step 4.} By Step 3, we thus have 
\begin{align*}
B_2^K(s)= &\int_0^1 \E\Big[A_2^K(W^{1}_s, W^*_s(\alpha), V^{1,N,K}_s, V^*_s(\bV^{N,K}_s,\bW_s,\alpha)) \\
& \hskip2cm-A_2^K(W^{1}_s, Z^*_s(\bW_s,\alpha), 
V^{1,N,K}_s, V^*_s(\bV^{N,K}_s,\bW_s,\alpha)) \Big] d\alpha.
\end{align*}
Consequently, Lemma \ref{further}-(ii) implies
\begin{align*}
B_2^K(s)\leq &C \int_0^1 \E\Big[ |W^{1}_s- V^{1,N,K}_s|^2+ |W^*_s(\alpha)-V^*_s(\bV^{N,K}_s,\bW_s,\alpha)|^2\\
& + |W^*_s(\alpha)- Z^*_s(\bW_s,\alpha)|^2
(1+|W^{1}_s|+|W^*_s(\alpha)|+|Z^*_s(\bW_s,\alpha)|  )^{2\gamma/(1-\gamma)} \Big] d\alpha.
\end{align*}

{\it Step 5.} Finally, we use Lemma \ref{further}-(iii) to obtain
\begin{align*}
B_3^K(s) \leq &C K^{1-2/\nu}\int_0^1 \E\Big[1+ |W^{1}_s|^{4\gamma/\nu+2}+ |W^*_s(\alpha)|^{4\gamma/\nu+2}
+ |V^{1,N,K}_s|^2 + |V^*_s(\bV^{N,K}_s,\bW_s,\alpha))|^2 \Big] d\alpha.
\end{align*}
Since $W^1_s\sim f_s$, we deduce from \eqref{momex} that $\E[|W^{1}_s|^{4\gamma/\nu+2}]= 
\intrd |v|^{4\gamma/\nu+2} f_s(dv) \leq C$.
By Lemma \ref{super}-(a), we also have $W^*_s(.)\sim f_s$, whence 
$\int_0^1 |W^*_s(\alpha)|^{4\gamma/\nu+2} d\alpha = \intrd |v|^{4\gamma/\nu+2} f_s(dv) \leq C$.
Proposition \ref{wpst} shows that $\E[|V^{1,N,K}_s|^2] = \intrd |v|^2f_0(dv)$.
We next infer from Lemma \ref{super}-(c) that $\int_0^1 |V^*_s(\bV^{N,K}_s,\bW_s,\alpha))|^2 d\alpha=
N^{-1}\sum_1^N |V^{i,N,K}_s|^2$. Consequently, 
$\E[\int_0^1 |V^*_s(\bV^{N,K}_s,\bW_s,\alpha))|^2 d\alpha ]= \E[|V^{1,N,K}_s|^2] = \intrd |v|^2f_0(dv)$.
As a conclusion,
\begin{align*}
B_3^K(s) \leq &C K^{1-2/\nu}.
\end{align*}

{\it Step 6.} We set $u_t^{N,K}:=\E[ |W^{1}_t- V^{1,N,K}_t|^2]$. Using the previous steps, we see that
for all $M\geq 1$,
\begin{align*}
u^{N,K}_t \leq& Ct e^{-M^{q/\gamma}} + Ct K^{1-2/\nu} +  (M+C)\intot \big[ u^{N,K}_s 
+ \int_0^1 \E [|W^*_s(\alpha)-V^*_s(\bV^{N,K}_s,\bW_s,\alpha)|^2  ] d\alpha \big]ds \\
&+C \intot   \int_0^1 \E\Big[|W^*_s(\alpha)- Z^*_s(\bW_s,\alpha)|^2
(1+|W^{1}_s|+|W^*_s(\alpha)|+|Z^*_s(\bW_s,\alpha)|  )^{2\gamma/(1-\gamma)} \Big] d\alpha   ds.
\end{align*}
We now write, using Minkowski's inequality and Lemma \ref{super}-(d) and (e),
\begin{align}\label{tbru}
&\left[\int_0^1 \E\Big[|W^*_s(\alpha)-V^*_s(\bV^{N,K}_s,\bW_s,\alpha)|^2\Big] d\alpha\right]^{1/2}\\
\leq & \left[\int_0^1 \E\Big[|W^*_s(\alpha)-Z^*_s(\bW_s,\alpha)|^2\Big] d\alpha \right]^{1/2}
+\left[ \int_0^1 \E\Big[|Z^*_s(\bW_s,\alpha)-V^*_s(\bV^{N,K}_s,\bW_s,\alpha)|^2\Big] d\alpha \right]^{1/2} 
\nonumber\\
=& \E [\cW_2^2(f_s,\mu^N_{\bW_s})]^{1/2} + \left[\frac1N \sum_1^N \E[|W^i_s- V^{i,N,K}_s|^2 ] \right]^{1/2}  
\nonumber\\ 
=& \E[\cW_2^2(f_s,\mu^N_{\bW_s})]^{1/2} + (u^{N,K}_s)^{1/2} \nonumber
\end{align}
by exchangeability. We deduce that
\begin{align}\label{topcool3}
\int_0^1 \E [|W^*_s(\alpha)-V^*_s(\bV^{N,K}_s,\bW_s,\alpha)|^2  ] d\alpha 
\leq 2 \E [\cW_2^2(f_s,\mu^N_{\bW_s})]  + 2u^{N,K}_s.
\end{align}
Next, a simple computation
shows that for all $\e \in(0,1)$,
\begin{align}\label{topcool2}
&\int_0^1 \E\Big[|W^*_s(\alpha)- Z^*_s(\bW_s,\alpha)|^2 (1+|W^{1}_s|+|W^*_s(\alpha)|
+|Z^*_s(\bW_s,\alpha)|  )^{\frac{2\gamma}{1-\gamma}} \Big] d\alpha   \\
\leq& \int_0^1 \E\Big[|W^*_s(\alpha)- Z^*_s(\bW_s,\alpha)|^{2-\e} (1+|W^{1}_s|+|W^*_s(\alpha)|
+|Z^*_s(\bW_s,\alpha)|  )^{\frac{2\gamma}{1-\gamma}+\e} \Big] d\alpha  \nonumber \\
\leq& \left(\int_0^1 \E\Big[|W^*_s(\alpha)- Z^*_s(\bW_s,\alpha)|^{2}\Big]d\alpha\right)^{\frac{2-\e}2}\nonumber \\
&\hskip2cm \times
\left(\int_0^1 \E\Big[ (1+|W^{1}_s|+|W^*_s(\alpha)|
+|Z^*_s(\bW_s,\alpha)|  )^{\frac{4\gamma}{\e(1-\gamma)}+2} \Big] d\alpha \right)^{\frac\e2}\nonumber \\
\leq& C_\e \left( \E [\cW_2^2(f_s,\mu^N_{\bW_s})] \right)^{\frac{2-\e}2}.\nonumber
\end{align}
For the last inequality, we used Lemma \ref{super}-(e), the fact that by \eqref{momex},
$$
\E\left[|W^{1}_s|^{\frac{4\gamma}{\e(1-\gamma)}+2}  \right] 
= \int_0^1|W^*_s(\alpha)|^{\frac{4\gamma}{\e(1-\gamma)}+2} d\alpha = \intrd |v|^{\frac{4\gamma}{\e(1-\gamma)}+2}f_s(dv) 
\leq C_\e
$$
and that, by Lemma \ref{super}-(b)
\begin{align*}
\int_0^1 \E\Big[ |Z^*_s(\bW_s,\alpha)|^{\frac{4\gamma}{\e(1-\gamma)}+2} \Big] d\alpha=&
\E\Big[\frac 1N \sum_1^N 
|W_s^i|^{\frac{4\gamma}{\e(1-\gamma)}+2} \Big] = \E\left[|W^{1}_s|^{\frac{4\gamma}{\e(1-\gamma)}+2}  \right]  \leq C_\e.
\end{align*}
We end up with: for all $\e\in(0,1)$, all $M\geq 1$,
\begin{align*}
u^{N,K}_t \leq& Ct e^{-M^{q/\gamma}} + Ct K^{1-2/\nu} +  3(M+C)\intot \big[ u^{N,K}_s +\E [\cW_2^2(f_s,\mu^N_{\bW_s})] 
\big] ds \\
&+C_\e \intot ( \E [\cW_2^2(f_s,\mu^N_{\bW_s})] )^{1-\e/2} ds.
\end{align*}
Now we observe that $\E [\cW_2^2(f_s,\mu^N_{\bW_s})]=\e_N(f_t)$, recall \eqref{bestrate}, because
$W^1_t,\dots,W^N_t$ are i.i.d. and $f_t$-distributed.
Since $\e_N(f_t)\leq 2 \intrd |v|^2f_t(dv)=2\intrd |v|^2f_0(dv)$, since 
$M\geq 1$ and $K\in [1,\infty)$, we get
\begin{align*}
u^{N,K}_t \leq& C_\e \left( t e^{-M^{q/\gamma}} + M t\delta_{N,K,t}^{1-\e/2}  +  M \intot u^{N,K}_s ds\right).
\end{align*}
where we have set
\begin{align*}
\delta_{N,K,t} := K^{1-2/\nu} + \sup_{[0,t]} \e_N(f_s).
\end{align*}
Hence by Gr\"onwall's lemma,
\begin{align*}
\sup_{[0,T]} u^{N,K}_t \leq& C_\e T\left( e^{-M^{q/\gamma}} + M \delta_{N,K,T}^{1-\e/2}\right) e^{C_\e M T},
\end{align*}
this holding for any value of $M\geq 1$. We easily conclude that
\begin{align*}
\sup_{[0,T]} u^{N,K}_t \leq& C_{\e,T} \delta_{N,K,T}^{1-\e},
\end{align*}
by choosing $M=1$ if $\delta_{N,K,T} \geq 1/e$
and $M=|\log \delta_{N,K,T}|^{\gamma/q}$ otherwise, which gives
\begin{align*}
\sup_{[0,T]} u^{N,K}_t \leq& C_\e \left( T \delta_{N,K,T} + \delta_{N,K,T}^{1-\e/2} |\log \delta_{N,K,T}|^{\gamma/q}    
\right) e^{C_\e |\log \delta_{N,K,T}|^{\gamma/q} T}\leq  C_{\e,T} \delta_{N,K,T}^{1-\e},
\end{align*}
the last inequality following from the fact that $\gamma/q<1$. 

\vip

{\it Final step.} We now recall that $\mu^{N,K}_t=\mu^N_{\bV^{N,K}_t}$ and write
$$
\E[\cW_2^2(\mu^{N,K}_t,f_t)]\leq 2 \E[\cW_2^2(\mu^N_{\bV^{N,K}_t},\mu^N_{\bW_t})]
+ 2 \E[\cW_2^2(\mu^N_{\bW_t},f_t)].
$$
But $\E[\cW_2^2(\mu^N_{\bV^{N,K}_t},\mu^N_{\bW_t})]\leq \E[N^{-1}\sum_1^N |V^{i,N,K}_t-W_t^i|^2] =
\E[|V^{1,N,K}_t-W_t^1|^2]=u^{N,K}_t$ by exchangeability,
and we have already seen that $\E[\cW_2^2(\mu^N_{\bW_t},f_t)]=\e_N(f_t)$. Consequently,
for all $\e\in(0,1)$, all $t\in [0,T]$,
$$
\E[\cW_2^2(\mu^N_t,f_t)]\leq C_{\e,T}\delta_{N,K,T}^{1-\e} + 2\e_N(f_t)
\leq C_{\e,T} \left( K^{1-2/\nu}  + \sup_{[0,T]} \e_N(f_t) \right)^{1-\e}
$$
and this proves \eqref{fc3}. Using finally \eqref{momex} and applying Theorem \ref{theo:W2indep}
(with any choice of $k>4$), \eqref{fc4} easily follows.
\end{preuve}

We next study the case of Maxwell molecules.

\begin{preuve} {\it of Theorem \ref{mr}-(i) when $K\in [1,\infty)$.}
We thus assume \eqref{cs}, \eqref{c1} with $\gamma=0$ and \eqref{c2}. We consider $f_0\in \cP_2(\rd)$
not being a Dirac mass. We also assume that $f_0\in \cP_4(\rd)$ or that $\intrd f_0(v)\log f_0(v)dv<\infty$,
so that $f_t$ has a density for all $t>0$.
We fix $N\geq 1$ and $K\in [1,\infty)$ and consider the processes introduced in Lemma \ref{coupling}.

\vip

{\it Step 1.} Exactly as in the case of hard potentials, we find that
\begin{align*}
\E[|W^{1}_t -V^{1,N,K}_t |^2] \leq \intot [B_1^K(s)+B_2^K(s)+B_3^K(s) ]ds,
\end{align*}
where  $B_i^K(s):= \int_0^1 \E\Big[A_i^K(W^{1}_s, W^*_s(\alpha), V^{1,N,K}_s, V^*_s(\bV^{N,K}_s,\bW_s,\alpha)) \Big] 
d\alpha$ for $i=1,2,3$. 

\vip

{\it Step 2.} By Lemma \ref{furthermax}-(i), we have $B^K_1(s)=0$.

\vip

{\it Steps 3 and 4.} By Lemma \ref{furthermax}-(ii), it holds that for $\zeta_K=\pi\int_0^K(1-\cos G(z))dz$,
\begin{align*}
B^K_2(s)= \zeta_K \int_0^1 \E\Big[-|W^{1}_s- V^{1,N,K}_s|^2+ |W^*_s(\alpha)-V^*_s(\bV^{N,K}_s,\bW_s,\alpha)|^2\Big] 
d\alpha.
\end{align*}

\vip

{\it Step 5.} By Lemma \ref{furthermax}-(iii)
\begin{align*}
B^K_3(s) \leq C K^{1-2/\nu} \int_0^1 \E\Big[ |W^{1}_s|^2 + | W^*_s(\alpha)|^2+ |V^{1,N,K}_s |^2 \Big] 
d\alpha \leq C K^{1-2/\nu},
\end{align*}
since, as usual, 
$\E[|W^1_s|^2]=\int_0^1 |W^*_s(\alpha)|^2 d\alpha= \E[|V^{1,N,K}_s|^2]=\intrd |v|^2f_0(dv)$.

\vip

{\it Step 6.} Setting $u_t^{N,K}:=\E[ |W^{1}_t- V^{1,N,K}_t|^2]$, we thus have
\begin{align*}
u^{N,K}_t \leq& C K^{1-2/\nu} t + \zeta_K \intot \left(-u^{N,K}_s +
\int_0^1 \E\Big[|W^*_s(\alpha)-V^*_s(\bV^{N,K}_s,\bW_s,\alpha)|^2\Big] d\alpha
 \right) ds \\
\leq & C K^{1-2/\nu} t + \zeta_K \intot \left(2\sqrt{u^{N,K}_s}\sqrt{\E\left[\cW_2^2(f_s,\mu^N_{\bW_s})\right]} 
+\E\left[\cW_2^2(f_s,\mu^N_{\bW_s})\right]\right) ds
\end{align*}
by \eqref{tbru}. Next we recall that $\e_N(f_t)= \E[\cW_2^2(f_s,\mu^N_{\bW_s})]$, we set
$\e_{N,T}=\sup_{[0,T]}\e_N(f_t)$ and we 
recall that $\zeta_K\leq \int_0^\infty (1-\cos G(z))dz<\infty$. We thus may write,
for all $t \in [0,T]$,
\begin{align*}
u^{N,K}_t  \leq & C (K^{1-2/\nu} + \e_{N,T}T)T + C \e_{N,T}^{1/2} \intot (u^{N,K}_s)^{1/2}ds =: v^{N,K}_t.
\end{align*}
Then we have $(v_t^{N,K})' \leq  C \e_{N,T}^{1/2}(v_t^{N,K})^{1/2}$, so that
$(v_t^{N,K})^{1/2}\leq (C (K^{1-2/\nu} + \e_{N,T}T) T)^{1/2} + C \e_{N,T}^{1/2}t$.
We conclude that
$$
\sup_{[0,T]}u^{N,K}_t \leq C (K^{1-2/\nu} + \e_{N,T} T) T + C T^2 \e_{N,T}
\leq C K^{1-2/\nu} T + C (T+T^2) \e_{N,T}
.
$$

{\it Final step.} Exactly as in the case of hard potentials, for $t\in[0,T]$,
\begin{align*}
\E[\cW_2^2(\mu^{N,K}_t,f_t)]\leq 2 \e_N(f_t)+ 2 u^{N,K}_t \leq
C  K^{1-2/\nu} T+ C (1+T)^2 \sup_{[0,T]}\e_N(f_t)
\end{align*}
whence \eqref{fc1}.
If finally $f_0 \in \cP_k(\rd)$ for all $k >4$, 
then we know that $\sup_{[0,\infty)} \intrd |v|^kf_t(dv)<\infty$,
so that \eqref{fc2} follows by application of 
Theorem~\ref{theo:W2indep}. 
\end{preuve}

We conclude with hard spheres.

\begin{preuve} {\it of Theorem \ref{mr}-(iii).}
We thus assume \eqref{cs}, \eqref{c1} with $\gamma=1$ and \eqref{c2hs}. We consider $f_0\in \cP_2(\rd)$
satisfying \eqref{c5} for some $p\in(\gamma,2)$ and fix $q\in(\gamma,p)$ 
for the rest of the proof.
We also assume that $f_0$ has a density, so that $f_t$ has a density for all $t>0$.
We fix $N\geq 1$ and $K\in [1,\infty)$ and consider 
the processes introduced in Lemma \ref{coupling}.

{\it Step 1.} Exactly as in the case of hard potentials, we find that
\begin{align*}
u^{N,K}_t:=\E[|W^{1}_t -V^{1,N,K}_t |^2] \leq \intot [B_1^K(s)+B_2^K(s)+B_3^K(s) ]ds,
\end{align*}
where  
$B_i^K(s):= \int_0^1 \E\Big[A_i^K(W^{1}_s, W^*_s(\alpha), V^{1,N,K}_s, V^*_s(\bV^{N,K}_s,\bW_s,\alpha)) \Big] 
d\alpha$ for $i=1,2,3$. 

\vip

{\it Steps 2, 3, 4, 5, 6.} Following the case of hard potentials, using Lemma \ref{furtherhs} instead of Lemma
\ref{further}, we deduce that for all $M>1$,
\begin{align*}
\sum_1^3B_i^K(s)\leq& 
2M \int_0^1 \E\left[|W^{1}_s- V^{1,N,K}_s|^2+ |W^*_s(\alpha)-V^*_s(\bV^{N,K}_s,\bW_s,\alpha)|^2\right] d\alpha \\
&+ C (K e^{-M^q}+e^{-K^q}) 
\int_0^1 \E\Big[(1+|V^{1,N,K}_s|+ |V^*_s(\bV^{N,K}_s,\bW_s,\alpha)|) \\
&\hskip5cm \times e^{C(|W^{1}_s|^q+|W^*_s(\alpha)|^q
+|Z^*_s(\bW_s,\alpha)|^q)}\Big] d\alpha \\
&+ C \int_0^1 \E\Big[|W^*_s(\alpha)-Z^*_s(\bV^{N,K}_s,\bW_s,\alpha)|^2\\
&\hskip3cm \times(1+|W^1_s|+|W^*_s(\alpha)|+
|Z^*_s(\bV^{N,K}_s,\bW_s,\alpha)|)^2\Big] d\alpha 
\end{align*}
Proceeding as in \eqref{topcool2}, we deduce that the last line is bounded, for all $\e\in(0,1)$,
by 
$$
C_\e \left(\E\left[\cW_2^2(f_s,\mu^N_{\bW_s})   \right]\right)^{\frac {2-\e}\e}
$$
and using \eqref{topcool3}, the first term is bounded by
$$
4M \E\left[\cW_2^2(f_s,\mu^N_{\bW_s})\right] + 6 M u^{N,K}_s.
$$
Using finally the Cauchy-Schwarz inequality, that, thanks to Lemma \ref{super}-(c) and by exchangeability,
$\E[\int_0^1 
|V^*_s(\bV^{N,K}_s,\bW_s,\alpha)|^2d\alpha]=\E[N^{-1}\sum_1^N |V^{i,N,K}_s|^2]=
\E[ |V^{1,N,K}_s|^2]=\intrd |v|^2 f_0(dv)<\infty$
and \eqref{momex}, we easily bound the second line by $C (K e^{-M^q}+e^{-K^q})$ (recall that
$W^i_s \sim f_s$, that $W^*_s(.) \sim f_s$ and that, by Lemma \ref{coupling}-(b), $\int_0^1 
e^{C|Z^*_s(\bW_s,\alpha)|^q} d\alpha= N^{-1} \sum_1^N e^{C|W^i_s|^q}$).

\vip

Recalling that  $\E\left[\cW_2^2(f_s,\mu^N_{\bW_s})\right]=\e_N(f_s)$ and setting
$\e_{N,t}=\sup_{[0,t]} \e_N(f_s)$, we thus have, for any $M>1$, any $\e\in(0,1)$,
\begin{align*}
u^{N,K}_t \leq & 6M \intot u^{N,K}_s ds + C t (K e^{-M^q}+e^{-K^q}) + C_\e t \e_{N,t}^{1-\e/2}.
\end{align*}
Thus by Gr\"onwall's Lemma,
\begin{align*}
u^{N,K}_t \leq & C_\e t (K e^{-M^q}+e^{-K^q}+ \e_{N,t}^{1-\e/2}) e^{6Mt}.
\end{align*}
Choosing $M=2K$ and using that $K e^{-(2K)^q} \leq C e^{-K^q}$, we deduce that
\begin{align*}
\sup_{[0,T]} u^{N,K}_t \leq & C_\e T (e^{-K^q}+ \e_{N,T}^{1-\e/2}) e^{12K T}= 
C_\e T (e^{-K^q}+ (\sup_{[0,T]} \e_N(f_s))^{1-\e/2}) e^{12K T}.
\end{align*}
\vip

{\it Final step.} We conclude as usual, using that 
$\E[\cW_2^2(\mu^{N,K}_t,f_t)]\leq 2 \e_N(f_t)+ 2 u^{N,K}_t$ to obtain \eqref{fc5} and then \eqref{momex}
and Theorem \ref{theo:W2indep} to deduce \eqref{fc6}.
\end{preuve}

\section{Extension to the particle system without cutoff}
\setcounter{equation}{0}
\label{sec:ConvGal}

It remains to check that the particle system without cutoff is well-posed and
that we can pass to the limit as $K\to \infty$ in the convergence estimates 
\eqref{fc1}-\eqref{fc2}-\eqref{fc3}-\eqref{fc4}.
We will need the following rough computations.

\begin{lem}\label{furthernul}
Assume \eqref{cs}, \eqref{c1} and \eqref{c2hs} or \eqref{c2}. 
Adopt the notation of Lemma \ref{fundest}.
There are $C>0$, $\kappa>0$ and $\delta>0$ 
(depending on $\gamma,\nu$) such that for all $K\in [1,\infty)$,
all $v,v_*,\tv,\tv_*\in\rd$,
$$
\sum_{i=1}^3 A_i^K(v,v_*,\tv,\tv_*) \leq C(1+|v|+|v_*|+|\tv|+|\tv_*|)^\kappa(|v-\tv|^2+|v_*-\tv_*|^2+K^{-\delta}).
$$
\end{lem}

\begin{proof}
Concerning $A_1^K$, we start from \eqref{ar1} (this is valid for all $\gamma\in[0,1]$) and we deduce that
\begin{align*}
A^K_1(v,v_*,\tv,\tv_*)\leq& 8c_4(|v-\tv|\land|v_*-\tv_*|)^\gamma (|v-\tv|+|v_*-\tv_*|)^2\\
\leq& C(1+|v|+|v_*|+|\tv|+|\tv_*|)^\gamma(|v-\tv|^2+|v_*-\tv_*|^2).
\end{align*}
We then make use of \eqref{fi} (also valid for all $\gamma \in [0,1]$) to write
\begin{align*}
A^K_2(v,v_*,\tv,\tv_*)\leq& C (|v-\tv|+|v_*-\tv_*|)^{2} (|v-v_*|^\gamma+|\tv-\tv_*|^{\gamma})\\
\leq& C(1+|v|+|v_*|+|\tv|+|\tv_*|)^\gamma(|v-\tv|^2+|v_*-\tv_*|^2).
\end{align*}
For $A^K_3$, we separate two cases. Under  hypothesis  \eqref{c2}, we immediately deduce from \eqref{ar2} that
\begin{align*}
A^K_3(v,v_*,\tv,\tv_*)
\leq& C(1+|v|+|v_*|+|\tv|+|\tv_*|)^{2+2\gamma/\nu} K^{1-2/\nu}.
\end{align*}
Under hypothesis \eqref{c2hs}, we have seen (when $\gamma=1$, at the end of the proof of Lemma 
\ref{furtherhs}) 
that $\Psi_K(x) \leq 5 x^\gamma \indiq_{\{x^\gamma\geq K/2\}}$, whence  $\Psi_K(x) \leq 10 x^{2\gamma} /K$ and thus
\begin{align*}
A^K_3(v,v_*,\tv,\tv_*)\leq& C (|v-v_*|\lor|\tv-\tv_*|)^{2+2\gamma} K^{-1}
\leq C(1+|v|+|v_*|+|\tv|+|\tv_*|)^{2+2\gamma} K^{-1}.
\end{align*}
The conclusion follows, choosing $\kappa=2+2\gamma/\nu$ and $\delta=2/\nu-1$ 
under  \eqref{c2} and $\kappa=2+2\gamma$ and $\delta=1$ under \eqref{c2hs}.
\end{proof}

Now we can give the

\begin{preuve} {\it of Proposition \ref{wps}-(ii).}
We only sketch the proof, since it is quite standard. In the whole proof, $N\geq 2$ is fixed,
as well as $f_0\in \cP_2(\rd)$ and a family of i.i.d. $f_0$-distributed random variables
$(V_0^{i,N})_{i=1,\dots,N}$.

\vip

{\it Step 1.} Recall \eqref{lgood}.
Classically, $(V_t^{i,N,\infty})_{i=1,\dots,N,t\geq 0}$ is a Markov process with generator $\cL_N$
starting from $(V_0^{i,N})_{i=1,\dots,N}$ if it solves
\begin{align}\label{pssdeinfty}
V^{i,N,\infty}_t=V^i_0 + \intot\int_j\int_0^\infty\int_0^{2\pi} c(V^{i,N,\infty}_\sm,V^{j,N,\infty}_\sm,z,\varphi)
O^N_i(ds,dj,dz,d\varphi),
\quad i=1,\dots,N
\end{align}
for some i.i.d. Poisson measures $O^N_i(ds,dj,dz,d\varphi))_{i=1,\dots,N}$ on
$[0,\infty)\times \{1,\dots,N\} \times[0,\infty)\times [0,2\pi)$ with intensity measures
$ds \left(N^{-1}\sum_{k=1}^N \delta_k(dj)\right) dz d\varphi$.

\vip

{\it Step 2.} The existence of a solution (in law) to \eqref{pssdeinfty} is easily checked, using
martingale problems methods (tightness and consistency),
by passing to the limit in \eqref{pssde}. The main estimates to be used are that,
uniformly in $K \in [1,\infty)$ (and in $N\geq 1$ but this is not the point here),
$$
\E[|V^{1,N,K}_t|^2]= \intrd |v|^2 f_0(dv) \quad \hbox{and} \quad
\E \left[\sup_{[0,T]} |V^{1,N,K}_t|\right] \leq C_T
$$
for all $T>0$. This second estimate is immediately deduced from the first one and the fact that
$\int_0^\infty \int_0^{2\pi} |c(v,v_*,z,\varphi)| \leq C |v-v_*|^{1+\gamma}\leq C(1+|v|+|v_*|)^2$. The tightness
is easily checked by using Aldous's criterion \cite{aldous}.

\vip

{\it Step 3.} Uniqueness (in law) for \eqref{pssdeinfty} is more difficult. Consider a
(c\`adl\`ag and adapted) solution
$(V_t^{i,N,\infty})_{i=1,\dots,N,t\geq 0}$ to \eqref{pssdeinfty}. For $K\in [1,\infty)$,
consider the solution to
\begin{align*}
V^{i,N,K}_t=V^i_0 + \intot\int_j\int_0^\infty\int_0^{2\pi} c_K(V^{i,N,K}_\sm,V^{j,N,K}_\sm,z,\varphi+\varphi_{s,i,j})
O^N_i(ds,dj,dz,d\varphi),
\quad i=1,\dots,N
\end{align*}
where $\varphi_{s,i,j}:=\varphi_0(V^{i,N,\infty}_\sm-V^{j,N,\infty}_\sm,V^{i,N,K}_\sm-V^{j,N,K}_\sm )$.
Such a solution obviously exists and is unique, because the involved Poisson measures are finite
(recall that $c_K(v,v_*,z,\varphi)=0$ for $z\geq K$). Furthermore, this solution
$(V_t^{i,N,K})_{i=1,\dots,N,t\geq 0}$ is a Markov process with generator $\cL_{N,K}$ starting from
$(V_0^{i,N})_{i=1,\dots,N}$ (because the only difference with \eqref{pssde} is the presence of
$\varphi_{s,i,j}$ which does not change the law of the particle system, see Lemma
\ref{coupling}-(ii) for a similar claim). Hence Proposition \ref{wps}-(i) implies that the law
of $(V_t^{i,N,K})_{i=1,\dots,N,t\geq 0}$ is uniquely determined.

We next introduce $\tau_{N,K,A}= \inf \{t\geq 0\; : \;
\exists \; i \in \{1,\dots,N\},\;  |V_t^{i,N,\infty}|+|V_t^{i,N,K}|\geq A\}$.
Using, on the one hand, the fact that $(V_t^{i,N,\infty})_{i=1,\dots,N,t\geq 0}$ is a.s. c\`adl\`ag 
(and thus locally bounded) and,
on the other hand, the (uniform in $K$) estimate established in Step 2, one easily gets convinced that
\begin{equation}\label{jab1}
\forall\; T >0, \quad \lim_{A\to \infty}\sup_{K\geq 1} \Pr[\tau_{N,K,A}\leq T ] = 0.
\end{equation}
Next, a simple computation shows that
\begin{align*}
\E[|V^{1,N,\infty}_{t\land \tau_{N,K,A}}-V^{1,N,K}_{t\land \tau_{N,K,A}}|^2] \leq &
\frac 1 N \sum_{j=1}^N\E\Big[ \int_0^{t\land \tau_{N,K,A}}  \int_0^\infty \int_0^{2\pi}
\Big(\big|V^{1,N,\infty}_\sm - V^{1,N,K}_\sm + \Delta^{1,j,N,K}_\sm(z,\varphi)\big|^2 - \\
&\hskip6cm \big|V^{1,N,\infty}_\sm - V^{1,N,K}_\sm \big|^2 \Big) d\varphi dz \Big]
\end{align*}
where
$$
\Delta^{1,j,N,K}_\sm(z,\varphi):= c(V^{1,N,\infty}_\sm,V^{1,N,\infty}_\sm,z,\varphi)
-c_K(V^{1,N,K}_\sm,V^{1,N,K}_\sm,z,\varphi+\varphi_{s,i,j}).
$$
Using Lemmas \ref{fundest} and \ref{furthernul}
and the fact that all the velocities are bounded by $A$
until $\tau_{N,K,A}$, we easily deduce that
\begin{align*}
&\E[|V^{1,N,\infty}_{t\land \tau_{N,K,A}}-V^{1,N,K}_{t\land \tau_{N,K,A}}|^2] \\
\leq&
\frac {C (1+A)^{\kappa}} N \sum_{j=1}^N\E\Big[ \int_0^{t\land \tau_{N,K,A}}
(|V^{1,N,\infty}_s - V^{1,N,K}_s|^2+|V^{j,N,\infty}_s - V^{j,N,K}_s|^2 + K^{-\delta}) ds\Big]\\
\leq & C_T (1+A)^{\kappa} K^{-\delta}  + C (1+A)^{\kappa} \intot
\E[|V^{1,N,\infty}_{s\land \tau_{N,K,A}}-V^{1,N,K}_{s\land \tau_{N,K,A}}|^2]ds
\end{align*}
by exchangeability. We now use the Gr\"onwall lemma and then deduce that for any $A>0$,
\begin{equation}\label{jab2}
\lim_{K\to \infty} \sup_{[0,T]}\E[|V^{1,N,\infty}_{t\land \tau_{N,K,A}}-V^{1,N,K}_{t\land \tau_{N,K,A}}|^2 ]=0.
\end{equation}
Gathering \eqref{jab1} and \eqref{jab2}, we easily conclude that
for all $t\geq 0$, $V^{1,N,K}_{t}$ tends in probability to  $V^{1,N,\infty}_{t}$ as $K\to \infty$.
Thus for any finite family $0\leq t_1 \leq \dots \leq t_l$, $(V^{i,N,K}_{t_j})_{i=1,\dots,N,j=1,\dots,l}$
goes in probability to  $(V^{i,N,\infty}_{t_j})_{i=1,\dots,N,j=1,\dots,l}$, of which the law is
thus uniquely determined. This is classically sufficient to characterize the whole law of
the process $(V^{i,N,\infty}_{t})_{i=1,\dots,N,t\geq 0}$.
\vip

{\it Conclusion.} We thus have the existence of a unique Markov process $(V_t^{i,N,\infty})_{i=1,\dots,N,t\geq 0}$
with generator $\cL_N$ starting from $(V_0^{i,N})_{i=1,\dots,N}$, and it holds that for each $t\geq 0$,
each $N\geq 2$, $(V_t^{i,N,\infty})_{i=1,\dots,N}$ is the limit in law, as $K\to \infty$, of
$(V_t^{i,N,K})_{i=1,\dots,N}$.
\end{preuve}

To conclude, we will need the following lemma.

\begin{lem}\label{fatouw2}
Let $N\geq 2$ be fixed.
Let $(X^{i,N,K})_{i=1,\dots,N}$ be a sequence of $(\rd)^N$-valued random variable going in law,
as $K\to \infty$, to some $(\rd)^N$-valued random variable $(X^{i,N})_{i=1,\dots,N}$.
Consider the associated empirical measures $\nu^{N,K}:=N^{-1}\sum_{i=1}^N \delta_{X^{i,N,K}}$ and
$\nu^{N}:=N^{-1}\sum_{i=1}^N \delta_{X^{i,N}}$. Then for any $g\in \cP_2(\rd)$,
$$
\E\left[\cW_2^2\left(\nu^N,g \right)\right] \leq
\liminf_{K \to \infty}\E\left[\cW_2^2\left(\nu^{N,K},g \right)\right].
$$
\end{lem}

\begin{proof}
First observe that the map $(x_1,\dots,x_N)\mapsto \cW_2(N^{-1}\sum_1^N\delta_{x_i},g)$ is
continuous on $(\rd)^N$. Indeed, it suffices to use the triangular inequality for $\cW_2$ and the
easy estimate
$$
\cW_2^2\left(\frac1N\sum_1^N\delta_{x_i},\frac1N\sum_1^N\delta_{y_i}\right)
\leq \frac1N \sum_1^N |x_i-y_i|^2.
$$
Consequently, $\cW_2^2\left(\nu^{N,K},g \right)$ goes in law to $\cW_2^2\left(\nu^{N},g \right)$.
Thus for any $A>1$, we have
$$
\E\left[\cW_2^2\left(\nu^{N},g \right) \land A\right]
=\lim_{K\to\infty} \E\left[\cW_2^2\left(\nu^{N,K},g \right) \land A\right]\leq
\liminf_{K \to \infty}\E\left[\cW_2^2\left(\nu^{N,K},g \right)\right].
$$
It then suffices to let $A$ increase to infinity and to use the monotonic convergence theorem.
\end{proof}

This allows us to conclude the proof of our main results.

\begin{preuve} {\it of Theorem \ref{mr}-(i)-(ii) when $K=\infty$.}
Recall that \eqref{fc1}-\eqref{fc2}-\eqref{fc3}-\eqref{fc4} 
have already been established when $K\in [1,\infty)$.
Since $(V^{i,N,\infty}_t)_{i=1,\dots,N}$ is the limit (in law) of $(V^{i,N,K}_t)_{i=1,\dots,N}$ as $K\to\infty$
for each $t\geq 0$ and each $N\geq 2$ (see the conclusion of the proof of Proposition \ref{wps}-(ii)),
we can let $K\to\infty$ in \eqref{fc1}-\eqref{fc2}-\eqref{fc3}-\eqref{fc4} using Lemma \ref{fatouw2}.
\end{preuve}

\def\refname{References}

\end{document}